\documentclass{article}
	\usepackage{amsfonts}
	\usepackage{amsmath}
	\usepackage{amsthm}
	\newtheorem{thm}{Theorem}[section]
	\newtheorem{cor}[thm]{Corollary}
	\newtheorem{lem}[thm]{Lemma}
	
	\theoremstyle{definition}
	\newtheorem{defn}[thm]{Definition}
	\newtheorem{rem}[thm]{Remark}
	\newtheorem{ex}{Example}
	\numberwithin{equation}{section}
	\usepackage{cite}
	\usepackage{authblk}
	\date{}
	\begin{document}
		\title
		{The Uncertainty Principles of Quaternion Fractional Fourier Transform }
		\author{Ke Cui$^{1}$, Haipan Shi\thanks{Corresponding author. E-mail: shihaipan226@163.com} and Xiaomin Tang$^{3}$ }
		\maketitle
		\begin{abstract}
		In this paper, we mainly establish the uncertainty principle (UP) for a function and its quaternion Fractional Fourier transform (QFrFT), as well as the UP for two QFrFTs. Using the polar representation of quaternion-valued signals, we give the UP for QFrFT in both the spatial and directional domains, providing a more precise condition for equality, example is given to verify the results. Furthermore, we extend the time-frequency UP to a frequency-frequency setting.

		\noindent \textbf{Keywords:} Uncertainty Principles, Quaternion 	Fractional Fourier transform, Covariance.
	\end{abstract}
	\section{Introduction}
The Uncertainty Principle (UP), proposed by Heisenberg \cite{8},
originates from the observation that the position and momentum of a particle cannot be simultaneously determined with  infinite precision. Later, Weyl \cite{15} showed that UP could be expressed by the standard deviation and presented a widely used proof relying on the Schwarz inequality. Then The UP was further generalized to arbitrary variables by Condon \cite{2}.
 Gabor\cite{5}  introduced it into the field of signal analysis, which led to lots of research in signal processing. In the view of physics, the UP reveals a fundamental law of the microscopic world which points out that some physical properties of a particle cannot be measured with unlimited precision at the same time. This principle ensures the stability of atomic structures and explains why quantum systems possess a lowest energy state.

 Mathematically, for the function \(f(t)\in L^{2}(\mathbb{R})\) with \(\|f\|_{L^{2}}=1,\) there holds
 \begin{equation}
 \sigma_t^2 \sigma_\omega^2 \ge \frac{1}{4},
 \end{equation}
 where \[
 \sigma_t^2 = \int_{-\infty}^{\infty} \left| (t - \langle t \rangle) f(t)\right| ^2 dt
 \]
 and
 \[
 \sigma_w^2 = \int_{-\infty}^{\infty} \left| (w - \langle w \rangle)\hat {f}(w)\right| ^2 dw
 \]
 denote the variances in the time and frequency domains, respectively. Here
 for time variable  \(t\) and the Fourier frequency  \(w\),
\[
\langle t \rangle := \int_{-\infty}^{\infty} t  \left| f(t)\right| ^2 \, dt
\]
and
\[
\langle w \rangle := \int_{-\infty}^{\infty} w  \left| \hat {f}(w)\right| ^2 \, dw
\]
represent the mean values of the \(t\)
 and  \(w\), respectively. Moreover,
 \(\hat {f}(w)\) is the Fourier transform (FT) of \(f(t)\).

For the complex signal \(f(t)\), which is known that \(f(t)=\rho (t)e^{i\theta(t)}\), Cohen\cite{3} gave a lower bound larger than (1.1), i.e.
\begin{equation}
\sigma_t^2 \sigma_\omega^2 \ge \frac{1}{4}+Cov^{2},
\end{equation}
where
\[
Cov: = \int_{-\infty}^{\infty}  \langle t\theta'(t) \rangle-\langle t \rangle\langle w \rangle =\int_{-\infty}^{\infty}t\theta'(t)\left| f(t)\right| ^{2}dt-\langle t \rangle\langle w \rangle
\]
is the covariance of  \(f(t)\). After that, Dang, Deng and Qian in \cite{4} gave a stronger result, i.e.
\begin{equation}
\sigma_t^2 \sigma_\omega^2 \ge \frac{1}{4}+COV^{2},
\end{equation}
where
\[
COV := \int_{-\infty}^{\infty}\left|  \bigl( t - \langle t \rangle) \bigl( \theta'(t) - \langle \omega \rangle)\right|  \left| f(t)\right| ^2 \, dt.
\]
 Clearly, the lower bound in (1.3) is larger than that in (1.2).

 Because the classical Fourier transform (FT) reflects a fixed time-frequency duality, it cannot effectively capture the more flexible time-frequency structures present in nonstationary signals. This limitation motivated the introduction of the fractional Fourier transform (FrFT), which can be viewed as a natural generalization of the FT.

For signals with multiple components, such as color images and polarized signals,  classic real- or complex-valued representations are often inadequate for simultaneously describing amplitude, phase, and multidimensional correlations. To address this issue, quaternions were introduced \cite{12} and have gradually become an effective mathematical framework for multi-component signal analysis \cite{13,14}. In particular, the quaternion Fourier transform (QFT) has been extensively employed in applications such as color image processing \cite{6,7}.
Furthermore, quaternions provide a unified mathematical representation for rotations and directional information in three-dimensional space.  In physics, many fundamental quantities, such as angular momentum, electromagnetic fields, and rigid body orientation, possess vectorial or rotational structures. Quaternions offer a concise and stable algebraic framework for the computation and transformation of these quantities.

  As quaternion signal theory has developed, UPs in the quaternion framework have attracted growing interest. Beginning in 1994, several studies \cite{1,9,11} investigated uncertainty relations related to the quaternion Fourier transform (QFT). The resulting bounds in these works were obtained without involving covariance terms. Then based on the polar representation of quaternion functions \(f(x)=\rho (x)e^{u(x)\theta(x)}\),  Yang, Dang, and Qian \cite{16} stated that a stronger form of the UP  for the QFT in both directional and spatial domains, i.e.

  \begin{equation}
  \int_{\mathbb{R}^2} |x|^2 |f(x)|^2 dx \int_{\mathbb{R}^2} |w|^2 |\mathcal{F}\{f\}(w)|^2 dw \ge 1 + \text{COV}_{xw}^2,
  \end{equation}
  where
  \[
  \text{COV}_{x\omega}:=\sum_{k=1}^{2}\int_{\mathbb{R}^2}\left| x_kNSc\left[ \left( \frac{\partial}{\partial x_k}e^{u(x)\theta(x)}\right)e^{u(x)\theta(x)}\right] \right|  \rho^2(x)dx.
  \]

  They strengthened the UP using absolute covariance for quaternion-valued signals, but did not consider QFrFT. Motivated by this, we extend the covariance-based framework to QFrFT, which allows a continuous rotation in the time-frequency plane and provides greater flexibility in signal representation.

   Compared with \cite{16}, our study deals with a two-sided noncommutative structure, making the analysis more challenging. In addition, the fractional-order transform we studied introduces an order parameter, which enables an arbitrary rotation in the time-frequency plane, and we explicitly characterize the equality condition of the uncertainty principle by identifying an appropriate Gaussian-type function. Our main results are stated as follows.
  \begin{thm} (Uncertainty principle in spatial case)
   	Let \(f(\textbf{x})=\rho (\textbf{x}) e^{\textbf{u}(\textbf{x}) \theta(\textbf{x})}\). If \(f(\textbf{x}) \in L^{2}(\mathbb{R}^{2}, \mathbb{H})\), \(x_{k} f(\textbf{x})\), \(\frac{\partial}{\partial x_{k}} f(\textbf{x}) \in L^{2}(\mathbb{R}^{2}, \mathbb{H})\)  and \(\|f\|_{L^{2}}=1\), Let \(e_{1}=i,e_{2}=j\), then
   \[
   \begin{aligned}
   &\int_{\mathbb{R}^{2}}\left| \textbf{x}\right| ^{2} |f(\textbf{x})|^{2} d \textbf{x}\int_{\mathbb{R}^{2}}\left| \textbf{w}\right| ^{2} |\mathcal{F}_{\alpha_{1},\alpha_{2}}\{f\}(\textbf{w})|^{2} d \textbf{w} \\
   \geq& \frac{1}{4} \left| \sin\alpha_{1}+\sin\alpha_{2}\right| ^{2}+ COV_{\textbf{x}\textbf{w},\alpha}^{2}+\int_{\mathbb{R}^{2}} \sum_{k=1}^{2} x_{k}^{2} \rho^{2} d\textbf{x} \int_{\mathbb{R}^{2}} \sum_{k=1}^{2}\cos^{2}\alpha_{k} x_{k}^{2} \rho^{2} d\textbf{x}
 \end{aligned}
   \]
   \begin{equation}
   \begin{aligned}
   &-\!\!\int_{\mathbb{R}^{2}} \!\sum_{k=1}^{2} x_{k}^{2} \rho^{2} d\textbf{x} \!\! \int_{\mathbb{R}^{2}}\!\sum_{k=1}^{2}2\sin\alpha_{k}\cos\alpha_{k}x_{k}\rho^{2}Sc\!\left[ e_{k}\!\!\left( \frac{\partial}{\partial x_{k}} e^{\textbf{u}(\textbf{x}) \theta(\textbf{x})} \right)\!\! \left( \!e^{-\textbf{u}(\textbf{x}) \theta(\textbf{x})} \!\right)\! \right] \!\!d\textbf{x},
   \end{aligned}
   \end{equation}
   where the absolute covariance is defined by
   \[
   COV_{\textbf{x}\textbf{w},\alpha}:=\sum_{k=1}^{2}\left| \sin\alpha_{k}\right| COV_{x_{k}w_{k}} .
   \]
   \(COV_{x_{k}w_{k}}\) is given in Theorem 4.7.
   The equality in (1.5) holds if and only if
   \[
   f(\textbf{x}) = \frac{1}{\sqrt{\pi\lambda\sin \alpha}}e^{-\frac{1}{2\lambda} \frac{1}{\sin\alpha}\left| \textbf{x}\right| ^2e^{\textbf{u}(\textbf{x})\theta (\textbf{x})}},
   \]
   where \( \left( \frac{\partial}{\partial x_k} e^{\textbf{u}(\textbf{x}) \theta(\textbf{x})} \right) e^{-\textbf{u}(\textbf{x}) \theta(\textbf{x})} = \beta x_k \). Here, \(\sin \alpha >0 \), \(\beta\)  is a pure quaternion.
    \end{thm}
In Theorem 1.1, we established the UP between the time domain and the QFrFT domain. Furthermore, the analysis is extended from the "time domain-transform domain"setting to the "transform domain-transform domain" setting, leading to a relation between QFrFT domains of different orders. This extension reveals the intrinsic connection between the geometric constraints associated with different fractional orders and the covariance structure, thereby rendering the uncertainty principle in the quaternion fractional framework more symmetric and complete.

\begin{thm}
		Let \(f(\textbf{x})=\rho (\textbf{x}) e^{\textbf{u}(\textbf{x}) \theta(\textbf{x})}\). If \(f(\textbf{x}) \in L^{2}(\mathbb{R}^{2}, \mathbb{H})\), \(x_{k} f(\textbf{x})\), \(\frac{\partial}{\partial x_{k}} f(\textbf{x}) \in L^{2}(\mathbb{R}^{2}, \mathbb{H})\)  and \(\|f\|_{L^{2}}=1\), then
	\[
	\begin{aligned}
	&\int_{\mathbb{R}^{2}}\left| \textbf{w}\right| ^{2} |\mathcal{F}_{\alpha_{1},\alpha_{2}}\{f\}(\textbf{w})|^{2} d \textbf{w}\int_{\mathbb{R}^{2}}\left| \textbf{y}\right| ^{2} |\mathcal{F}_{\beta_{1},\beta_{2}}\{f\}(\textbf{y})|^{2} d \textbf{y}\\
	\geq& \left| \int_{\mathbb{R}^{2}}  \sum_{k=1}^{2}sin \alpha_ksin \beta_k  \left( \frac{\partial}{\partial x_{k}} \rho(\textbf{x})\right)  ^{2}d\textbf{x}\right| ^{2}+\frac{1}{4}\left| \sum_{k=1}^{2}sin \alpha_kcos \beta_k\right| ^{2}\\
	&+\left| \int_{\mathbb{R}^{2}}  \sum_{k=1}^{2}sin \alpha_ksin \beta_k\rho(\textbf{x})  \left( \frac{\partial}{\partial x_{k}} \rho(\textbf{x})\right)  \left|  \left( \frac {\partial }{\partial x_{k}}e^{\textbf{u}(\textbf{x})\theta (\textbf{x})}\right) e^{-\textbf{u}(\textbf{x})\theta (\textbf{x})}\right|d\textbf{x}\right| ^{2}\\
	&+\frac{1}{4}\left| \sum_{k=1}^{2}cos \alpha_ksin \beta_k\right| ^{2}+\left| \int_{\mathbb{R}^{2}}  \sum_{k=1}^{2}cos \alpha_kcos \beta_k   x_{k}^{2}\rho^{2}(\textbf{x})d\textbf{x}\right| ^{2}\\
	&+\left| \int_{\mathbb{R}^{2}}  \sum_{k=1}^{2}cos \alpha_ksin \beta_kx_{k}\rho^{2}(\textbf{x}) \left|  \left( \frac {\partial }{\partial x_{k}}e^{\textbf{u}(\textbf{x})\theta (\textbf{x})}\right) e^{-\textbf{u}(\textbf{x})\theta (\textbf{x})}\right|d\textbf{x}\right| ^{2}\\
	&+\left| \int_{\mathbb{R}^{2}}  \sum_{k=1}^{2}sin \alpha_ksin \beta_k\rho(\textbf{x}) \left( \frac{\partial}{\partial x_{k}} \rho(\textbf{x})\right) \left|  \left( \frac {\partial }{\partial x_{k}}e^{\textbf{u}(\textbf{x})\theta (\textbf{x})}\right) e^{-\textbf{u}(\textbf{x})\theta (\textbf{x})}\right|d\textbf{x}\right| ^{2}\\
	&+ \left| \int_{\mathbb{R}^{2}}  \sum_{k=1}^{2}sin \alpha_kcos \beta_kx_{k}\rho^{2}(\textbf{x}) \left|  \left( \frac {\partial }{\partial x_{k}}e^{\textbf{u}(\textbf{x})\theta (\textbf{x})}\right) e^{-\textbf{u}(\textbf{x})\theta (\textbf{x})}\right|d\textbf{x}\right| ^{2}\\
	&+\left| \int_{\mathbb{R}^{2}}  \sum_{k=1}^{2}sin \alpha_ksin \beta_k\rho^{2}(\textbf{x}) \left|  \left( \frac {\partial }{\partial x_{k}}e^{\textbf{u}(\textbf{x})\theta (\textbf{x})}\right) e^{-\textbf{u}(\textbf{x})\theta (\textbf{x})}\right|^{2}d\textbf{x}\right| ^{2}\\
	&-I_{1}I_{8}-I_{2}I_{8}-I_{3}I_{8}-I_{4}I_{5}- I_{4}I_{6}-I_{4}I_{7}-I_{4}I_{8},
	\end{aligned}
	\]
	where
\[
\begin{aligned}
&I_{1}=\int_{\mathbb{R}^{2}} \sum_{k=1}^{2}sin^{2} \alpha_k  \left( \frac{\partial}{\partial x_{k}} \rho(\textbf{x})\right)  ^{2} d\textbf{x},
I_{2}=\int_{\mathbb{R}^{2}} \sum_{k=1}^{2}\cos^{2}\alpha_{k} x_{k}^{2} \rho^{2} d\textbf{x},
	\end{aligned}
\]
\[
\begin{aligned}
&I_{3}=\int_{\mathbb{R}^{2}} \sum_{k=1}^{2}sin^{2} \alpha_k \rho^{2}\left| \left( \frac{\partial}{\partial x_{k}} e^{\textbf{u}(\textbf{x}) \theta(\textbf{x})} \right) \left( e^{-\textbf{u}(\textbf{x}) \theta(\textbf{x})} \right) \right| ^{2}d\textbf{x},\\
&I_{4}=\int_{\mathbb{R}^{2}}\sum_{k=1}^{2}2\sin\alpha_{k}\cos\alpha_{k}x_{k}\rho^{2}Sc \left[ e_{k}\left( \frac{\partial}{\partial x_{k}} e^{\textbf{u}(\textbf{x}) \theta(\textbf{x})} \right) \left( e^{-\textbf{u}(\textbf{x}) \theta(\textbf{x})} \right) \right] d\textbf{x},\\
&I_{5}=\int_{\mathbb{R}^{2}} \sum_{k=1}^{2}sin^{2} \beta_k  \left( \frac{\partial}{\partial x_{k}} \rho(\textbf{x})\right)  ^{2} d\textbf{x},
I_{6}=\int_{\mathbb{R}^{2}} \sum_{k=1}^{2}\cos^{2}\beta_{k} x_{k}^{2} \rho^{2} d\textbf{x}, \\
&I_{7}=\int_{\mathbb{R}^{2}} \sum_{k=1}^{2}sin^{2} \beta_k \rho^{2}\left| \left( \frac{\partial}{\partial x_{k}} e^{\textbf{u}(\textbf{x}) \theta(\textbf{x})} \right) \left( e^{-\textbf{u}(\textbf{x}) \theta(\textbf{x})} \right) \right| ^{2}d\textbf{x},\\
&I_{8}=\int_{\mathbb{R}^{2}}\sum_{k=1}^{2}2\sin\beta_{k}\cos\beta_{k}x_{k}\rho^{2}Sc \left[ e_{k}\left( \frac{\partial}{\partial x_{k}} e^{\textbf{u}(\textbf{x}) \theta(\textbf{x})} \right) \left( e^{-\textbf{u}(\textbf{x}) \theta(\textbf{x})} \right) \right] d\textbf{x}.
\end{aligned}
\]
\end{thm}

  The article is organized as follows. In Section 2, we present the quaternion algebra and the polar form of quaternion-valued signals.
  In Section 3, we recall the definition and properties of QFrFT. In Section 4, the proof of the main theorem has been fully provided. In Section 5, we provide an example to verify the obtained results.

	\section{Preliminaries}

	The quaternion algebra \(\mathbb{H}\) generalizes complex numbers to a four-dimensional non-commutative system \cite{10}, which was introduced by  Hamilton. Any real quaternion $q \in \mathbb{H}$ may be expressed as
	\[
	q = q_0 + \underline{q} = q_0 + iq_1 + jq_2 + kq_3, q_m \in \mathbb{R}, m=0,1,2,3,
	\]
	where $i, j, k$ have Hamilton's multiplication rules
	\[
	i^2 = j^2 = k^2 = -1, \quad ij = -ji = k,
	\]
	\[
	jk = -kj = i, \quad ki = -ik = j.
	\]

	The scalar component of $q$ is denoted by  $q_0$, written as  $\text{Sc}[q] = q_0$.  The non-scalar component, also called the pure quaternion, is represented by  $\underline{q}$, denoted as $\text{NSc}[q] = \underline{q}$.
	
The multiplication of  $p = p_0 + \underline{p}$ and $q = q_0 + \underline{q}$ can be written as
	\[
	pq = p_0q_0 + \underline{p} \cdot \underline{q} + p_0\underline{q} + q_0\underline{p} + \underline{p} \times \underline{q},
	\]
	where
	\[
	\underline{p} \cdot \underline{q} = -\left(p_1q_1 + p_2q_2 + p_3q_3\right)
	\]
	and
	\[
	\underline{p} \times \underline{q} = i(p_3q_2 - p_2q_3) + j(p_1q_3 - p_3q_1) + k(p_2q_1 - p_1q_2).
	\]

	For a quaternion \(q\in\mathbb{H}\), its conjugate is defined as
 $\overline{q} = q_0 - iq_1 - jq_2 - kq_3$.
 The quaternion conjugation is a linear antinvolution satisfying
	\[
	\overline{\overline{q}} = q, \quad \overline{p+q} = \overline{p} + \overline{q}, \quad \overline{pq} = \overline{q} \,\overline{p}.
	\]
	
	Clearly, $q\overline{q} = \overline{q}q = q_0^2 + q_1^2 + q_2^2 + q_3^2$. Thus, the modulus of a quaternion is given by
	\[
	|q| = \sqrt{q\overline{q}} = \sqrt{q_0^2 + q_1^2 + q_2^2 + q_3^2}.
	\]
	
	For $0 \neq q \in \mathbb{H}$, the inverse quaternion can be verified as
	\begin{equation}
	q^{-1} = \frac{\overline{q}}{|q|^2}.
	\end{equation}
	
Here, we consider quaternion-valued signals $f: \mathbb{R}^2 \to \mathbb{H}$ with
	\[
	f(\textbf{x}) = f_0(\textbf{x}) + if_1(\textbf{x}) + jf_2(\textbf{x}) + kf_3(\textbf{x}),
	\]
	where $\textbf{x} = (x_1,x_2) \in \mathbb{R}^2$ and $f_m$ ($m=0,1,2,3$) are real-valued functions.
	
	
	The polar form of quaternion-valued signals is given by
	\[
	\begin{aligned}
	f(\textbf{x}) &= f_0(\textbf{x}) + if_1(\textbf{x}) + jf_2(\textbf{x}) + kf_3(\textbf{x})
	= |f(\textbf{x})|e^{\textbf{u}(\textbf{x})\theta(\textbf{x})}
	= \rho(\textbf{x})e^{\textbf{u}(\textbf{x})\theta(\textbf{x})},
	\end{aligned}
	\]
	where $e^{\textbf{u}(\textbf{x})\theta(\textbf{x})}$ follows Euler's formula $e^{\textbf{u}(\textbf{x})\theta(\textbf{x})} = \cos\theta(\textbf{x}) + \textbf{u}(\textbf{x})\sin\theta(\textbf{x})$, and
	\[
	\rho(\textbf{x}) := \sqrt{f_0^2(\textbf{x}) + f_1^2(\textbf{x}) + f_2^2(\textbf{x}) + f_3^2(\textbf{x})},
	\]
	\[
	\textbf{u}(\textbf{x}) := \frac{if_1(\textbf{x}) + jf_2(\textbf{x}) + kf_3(\textbf{x})}{\sqrt{f_1^2(\textbf{x}) + f_2^2(\textbf{x}) + f_3^2(\textbf{x})}}.
	\]
Here $\textbf{u}(\textbf{x})$
belongs to the unit sphere $S^2 := \{q \in \mathbb{H} \mid |q|^2 = 1\}$. The quaternionic phase
	\[
	\theta(\textbf{x}) := \arctan\frac{\sqrt{f_1^2(\textbf{x}) + f_2^2(\textbf{x}) + f_3^2(\textbf{x})}}{f_0(\textbf{x})} \in [0, \pi].
	\]

The inner product of $f(\textbf{x}), g(\textbf{x}) \in L^2(\mathbb{R}^2, \mathbb{H})$ be defined by
	\[
	\langle f(\textbf{x}), g(\textbf{x}) \rangle := \int_{\mathbb{R}^2} f(\textbf{x})\overline{g(\textbf{x})} d\textbf{x}.
	\]
	Clearly, $\|f\|_{L^2}^2 = \langle f, f \rangle$.

Here, we present three spaces used in this paper.
	\[
	L^{1}\left(\mathbb{R}^{2} ; \mathbb{H}\right)=\left\{f: \mathbb{R}^{2} \to \mathbb{H} \,\middle|\, \int_{\mathbb{R}^{2}} |f(\textbf{x})| \, d\textbf{x} < \infty\right\},
	\]
	\[
	L^{2}\left(\mathbb{R}^{2} ; \mathbb{H}\right)=\left\{f: \mathbb{R}^{2} \to \mathbb{H} \,\middle|\, \int_{\mathbb{R}^{2}} |f(\textbf{x})|^{2} \, d\textbf{x} < \infty\right\},
	\]
	\[
	S\left(\mathbb{R}^{2} ; \mathbb{H}\right)=\left\{f \in C^{\infty}\left(\mathbb{R}^{2} ; \mathbb{H}\right) \,\middle|\, \sup _{\textbf{x} \in \mathbb{R}^{2}}\left(1+|\textbf{x}|^{k}\right)\left|\partial^{\alpha} f(\textbf{x})\right| < \infty\right\},
	\]
	where \(C^{\infty}(\mathbb{R}^{2} ; \mathbb{H})\) is the set of all infinitely differentiable functions from \(\mathbb{R}^{2}\) to \(\mathbb{H}\), and \(\alpha = (\alpha_{1}, \alpha_{2})\) (with \(\alpha_1, \alpha_2 \in \mathbb{Z}_{+}\)) and \(k \in \mathbb{Z}_{+}\).

	\section{Some properties of the QFrFT}
	In this section we recall the definition and properties of QFrFT. We first give the definition of the QFrFT and its inverse transformation. Then we study derivative formula, Plancherel identity, Parseval identity, etc.
		\begin{defn} Suppose the function \(f \in L^1(\mathbb{R}^2; \mathbb{H})\).  $p = (p_1, p_2)$-order two-sided QFrFT is defined by:
		\[
		\mathcal{F}_{\alpha_1,\alpha_2}\{f\}(\textbf{w}) = \int_{\mathbb{R}^2} K_{\alpha_1}(x_1,w_1)f(\textbf{x})K_{\alpha_2}(x_2,w_2) \, d\textbf{x},
		\]
		where the kernel is
		\[
		K_{\alpha_1}(x_1,w_1) = C_{\alpha_1} e^{i  \frac{x_1^2 + w_1^2}{2} \cot\alpha_1 - ix_1w_1 \csc\alpha_1}, \quad C_{\alpha_1} = \sqrt{\frac{1 - i \cot\alpha_1}{2\pi}},
		\]
		\[
		K_{\alpha_2}(x_2,w_2) = C_{\alpha_2} e^{j  \frac{x_2^2 + w_2^2}{2} \cot\alpha_2 - jx_2w_2 \csc\alpha_2 }, \quad C_{\alpha_2} = \sqrt{\frac{1 - j \cot\alpha_2}{2\pi}}
		\]
		and $\alpha_k \neq n\pi$, $p_k = \frac{2\alpha_k}{\pi}$ for $k = 1, 2$ and $n\in \mathbb{Z}$.
	\end{defn}
	As a special case, if $\alpha_1 = \alpha_2 = \dfrac{\pi}{2}+2n\pi$ for $n\in \mathbb{Z}$, then the two-sided QFrFT leads to the two-sided QFT.
	\begin{defn} Assume that \(f \in L^{1}(\mathbb{R}^{2} ; \mathbb{H})\). The inverse transformation  is defined as follows:
		\[
		\mathcal{H}_{\alpha_{1},\alpha_{2}}\{f\}(\textbf{w})=\int_{\mathbb{R}^{2}} K_{-\alpha_{1}}(x_{1},w_{1}) f(\textbf{x}) K_{-\alpha_{2}}(x_{2},w_{2}) \, d\textbf{x}, \tag{3.2}
		\]
		where
		\[
		K_{-\alpha_{1}}(x_{1},w_{1}) = C_{-\alpha_{1}} e^{-i \frac{x_{1}^{2}+w_{1}^{2}}{2} \cot\alpha_{1} + i x_{1} w_{1} \csc\alpha_{1}}, \quad C_{-\alpha_{1}} = \sqrt{\frac{1+i \cot\alpha_{1}}{2\pi}},
		\]
		\[
		K_{-\alpha_{2}}(x_{2},w_{2}) = C_{-\alpha_{2}} e^{-j \frac{x_{2}^{2}+w_{2}^{2}}{2} \cot\alpha_{2} + j x_{2} w_{2} \csc\alpha_{2}}, \quad C_{-\alpha_{2}} = \sqrt{\frac{1+j \cot\alpha_{2}}{2\pi}}
		\]
		and \(\alpha_{k}\ (k=1,2)\) are as mentioned above.
	\end{defn}
	\begin{lem}(\!\!\cite{10})(Plancherel Theorem)
		Suppose that \(f_{1}, f_{2} \in L^{2}(\mathbb{R}^{2} ; \mathbb{H})\), then we have
		\[
		Sc \langle\mathcal{F}_{\alpha_{1}, \alpha_{2}}\left\{f_{1}\right\}(\textbf{w}), \mathcal{F}_{\alpha_{1}, \alpha_{2}}\left\{f_{2}\right\}(\textbf{w})\rangle = 	Sc \langle f_{1}, f_{2} \rangle.
		\]
		In particular, with \(f_1=f_2=f\), we get the Parseval theorem, i.e.
		\[
		\| f\| ^{2}=\| \mathcal{F}_{\alpha_{1},\alpha_{2}}\{f\}(\textbf{w})\| ^{2}.
		\]
	\end{lem}
	\begin{lem}(\!\!\cite{10})
		Let the function \(f \in S(\mathbb{R}^{2} ; \mathbb{H})\), for each component \(x_{k}\) and \(w_{k}\) with \(k=1,2\), the following identity holds:
		\[
		\begin{aligned}
		\mathcal{F}_{\alpha_{1}, \alpha_{2}}\!\left\{\frac{\partial f}{\partial x_{k}}\right\}(\!\textbf{w}) \!\!
		& =\!\!
		\begin{cases}
		i \cot \alpha_{1} \mathcal{F}_{\alpha_{1}, \alpha_{2}}\left\{x_{1} f\right\}(\!\textbf{w}) \!+ \!i w_{1} \csc \alpha_{1} \mathcal{F}_{\alpha_{1}, \alpha_{2}}\{f\}(\!\textbf{w}), & \!\!k=1 ; \\
		-\cot \alpha_{2} \mathcal{F}_{\alpha_{1}, \alpha_{2}}\left\{x_{2} f\right\}(\!\textbf{w})j\! +\! w_{2}\! \csc \alpha_{2} \mathcal{F}_{\alpha_{1}, \alpha_{2}}\{f\}(\!\textbf{w}) j, &\!\! k=2 .
		\end{cases}
		\end{aligned}
		\]
	\end{lem}

 	\section{Proof of main results}

	In this section, we provide the proof of the main result. To this end, we first present some preliminary lemmas. Next, we introduce the UP in the directional domain. Finally, combining these preparations, we complete the proof of the main theorem.

	The following lemma shows that \(Sc\left[\left( \frac{\partial}{\partial x_{k}} e^{\textbf{u}(\textbf{x}) \theta(\textbf{x})}\right) \left( e^{-\textbf{u}(\textbf{x}) \theta(\textbf{x})}\right) \right] =0\), Then  \(\left( \frac{\partial}{\partial x_{k}} e^{\textbf{u}(\textbf{x}) \theta(\textbf{x})}\right) \left( e^{-\textbf{u}(\textbf{x}) \theta(\textbf{x})}\right)=NSc\left[ \left( \frac{\partial}{\partial x_{k}} e^{\textbf{u}(\textbf{x}) \theta(\textbf{x})}\right) \left( e^{-\textbf{u}(\textbf{x}) \theta(\textbf{x})}\right)\right] \). In the following article, all instances of \(\left( \frac{\partial}{\partial x_{k}} e^{\textbf{u}(\textbf{x}) \theta(\textbf{x})}\right) \left( e^{-\textbf{u}(\textbf{x}) \theta(\textbf{x})}\right)\) represent only its non-scalar part.
	
	\begin{lem}(\!\!\cite{16}) Suppose  a quaternion-valued signal \(f(\textbf{x}) = \rho(\textbf{x}) e^{\textbf{u}(\textbf{x}) \theta(\textbf{x})}\), if \(\frac{\partial \textbf{u}}{\partial x_{k}}\) and \(\frac{\partial \theta}{\partial x_{k}}\) exist for \(k=1,2\), then the scalar part of
	\[
	\left[\frac{\partial}{\partial x_{k}} e^{\textbf{u}(\textbf{x}) \theta(\textbf{x})}\right]\left[e^{-\textbf{u}(\textbf{x}) \theta(\textbf{x})}\right]
	\]
	is zero.
	\end{lem}

    \begin{lem}(\!\!\cite{16})
 Let \(f(\textbf{x}) = \rho(\textbf{x}) e^{\textbf{u}(\textbf{x}) \theta(\textbf{x})}\), if \(\frac{\partial}{\partial x_{k}} f(\textbf{x})\) exists for \(k=1,2\), then
\[
\left| \frac{\partial}{\partial x_{k}} f(\textbf{x}) \right|^{2} = \left(  \frac{\partial}{\partial x_{k}} \rho(\textbf{x}) \right) ^{2} + \rho^{2}(\textbf{x}) \left|  \left( \frac{\partial}{\partial x_{k}} e^{\textbf{u}(\textbf{x}) \theta(\textbf{x})} \right) \left( e^{-\textbf{u}(\textbf{x}) \theta(\textbf{x})} \right)  \right|^{2}.
\]
	\end{lem}

\begin{lem}
 Let \(f(\textbf{x}) = \rho(\textbf{x}) e^{\textbf{u}(\textbf{x}) \theta(\textbf{x})}\), if \(\frac{\partial \textbf{u}}{\partial x_{k}}\) and \(\frac{\partial \theta}{\partial x_{k}}\) exist for \(k=1,2\), let \(e_{1}=i,e_{2}=j\), we have
	\[
	Sc\left\lbrace e_{k}\frac{\partial}{\partial x_{k}} f(\textbf{x}) x_{k}\overline{f(\textbf{x})}\right\rbrace =Sc\left\lbrace e_{k}x_{k}\rho^{2}(\textbf{x})\left(\frac{\partial}{\partial x_{k}} e^{\textbf{u}(\textbf{x}) \theta(\textbf{x})}\right)\left(e^{-\textbf{u}(\textbf{x}) \theta(\textbf{x})}\right)\right\rbrace.
	\]
	Proof.
	Since \(\frac{\partial}{\partial x_{k}} f(\textbf{x})=\frac{\partial \rho (\textbf{x})}{\partial x_{k}} e^{\textbf{u}(\textbf{x}) \theta(\textbf{x})}+\rho (\textbf{x})\frac{\partial}{\partial x_{k}} e^{\textbf{u}(\textbf{x}) \theta(\textbf{x})}
	\) ,
	\(\overline{f(x)}=\rho(\textbf{x}) e^{-\textbf{u}(\textbf{x}) \theta(\textbf{x})}\), substituting the above two equations, we obtain
	\[
	\begin{aligned}
	e_{k} \frac{\partial}{\partial x_{k}} f(\textbf{x}) x_{k}\overline{f(\textbf{x})}
	&=e_{k}\left[ \frac{\partial \rho (\textbf{x})}{\partial x_{k}} e^{\textbf{u}(\textbf{x}) \theta(\textbf{x})}+\rho (\textbf{x})\frac{\partial}{\partial x_{k}} e^{\textbf{u}(\textbf{x}) \theta(\textbf{x})}\right] x_{k}\rho (\textbf{x})e^{-\textbf{u}(\textbf{x}) \theta(\textbf{x})}\\
	&=e_{k}x_{k}\rho (\textbf{x})\frac{\partial \rho (\textbf{x})}{\partial x_{k}}+e_{k}x_{k}\rho^{2}(\textbf{x})\left(\frac{\partial}{\partial x_{k}} e^{\textbf{u}(\textbf{x}) \theta(\textbf{x})}\right)\left(e^{-\textbf{u}(\textbf{x}) \theta(\textbf{x})}\right).
	\end{aligned}
	\]
		Obviously, the first term on the right-hand side of the equation is the none-scalar part, then
	\[
	Sc\left\lbrace e_{k}\frac{\partial}{\partial x_{k}} f(\textbf{x}) x_{k}\overline{f(\textbf{x})}\right\rbrace =Sc\left\lbrace e_{k}x_{k}\rho^{2}(\textbf{x})\left(\frac{\partial}{\partial x_{k}} e^{\textbf{u}(\textbf{x}) \theta(\textbf{x})}\right)\left(e^{-\textbf{u}(\textbf{x}) \theta(\textbf{x})}\right)\right\rbrace .
	\]
\end{lem}
\hfill $\square$
\begin{lem}
	 Let \(f(\textbf{x}) = \rho(\textbf{x}) e^{\textbf{u}(\textbf{x}) \theta(\textbf{x})}\) , if \(\frac{\partial}{\partial x_{k}} f(\textbf{x})\) exists for \(k=1,2\), then
	\[
	\begin{aligned}
&\int_{\mathbb{R}^{2}}\!\!i\mathcal{F}_{\alpha_{1},\alpha_{2}}\!\left\{\!\frac{\partial f}{\partial x_{1}}\!\right\}\!(\!\textbf{w}) \overline{\mathcal{F}_{\alpha_{1},\alpha_{2}}\!\!\left\{x_{1}f\right\}\!\!(\textbf{w})} d\textbf{w}  \!+\!\!\int_{\mathbb{R}^{2}}\!\!\mathcal{F}_{\alpha_{1},\alpha_{2}}\!\!\left\lbrace\! x_{1}f\right\rbrace\!\! (\!\textbf{w}) \overline{i\mathcal{F}_{\alpha_{1},\alpha_{2}}\!\!\left\{\frac{\partial f}{\partial x_{1}}\right\}\!\!(\!\textbf{w})} d\textbf{w}.\\
=&2\int_{\mathbb{R}^{2}}Sc\left[ i\mathcal{F}_{\alpha_{1},\alpha_{2}}\left\{\frac{\partial f}{\partial x_{1}}\right\}(\textbf{w}) \overline{\mathcal{F}_{\alpha_{1},\alpha_{2}}\left\{x_{1}f\right\}(\textbf{w})} \right] d\textbf{w}.
	\end{aligned}	
	\]
		Proof. Let \(f=i\mathcal{F}_{\alpha_{1},\alpha_{2}}\left\{\frac{\partial f}{\partial x_{1}}\right\}(\textbf{w}) , g=\mathcal{F}_{\alpha_{1},\alpha_{2}}\left\{x_{1}f\right\}(\textbf{w}) \), we know
	\[
\int_{\mathbb{R}^{2}}f\overline{g}d\textbf{w}+\int_{\mathbb{R}^{2}}g\overline{f}d\textbf{w}=\int_{\mathbb{R}^{2}}f\overline{g}+\overline{f\overline{g}}d\textbf{w}=2\int_{\mathbb{R}^{2}}Sc(f\overline{g})d\textbf{w}.
	\]
	Thus, we obtain the final result.
\end{lem}
\hfill $\square$

 Due to the non-commutative nature of quaternions, directly computing the QFrFT is complicated. The following theorem provides an effective way to evaluate \(\int_{\mathbb{R}^{2}} w_{k}^{2}|\mathcal{F}_{\alpha_{1},\alpha_{2}}\{f\}(\textbf{w})|^{2} d \textbf{w}\). Using this formula, we can greatly simplify the subsequent derivation of the UP.
    \begin{thm}
     Let \(f(\textbf{x}) = \rho(\textbf{x}) e^{\textbf{u}(\textbf{x}) \theta(\textbf{x})}\), if \(f \in L^{1}(\mathbb{R}^{2}, \mathbb{H}) \cap L^{2}(\mathbb{R}^{2}, \mathbb{H})\), and  \(\frac{\partial}{\partial x_{1}} f\),\(\frac{\partial}{\partial x_{2}} f\) exists and is also in \(L^{2}(\mathbb{R}^{2}, \mathbb{H})\). Let \(e_{1}=i,e_{2}=j\), then
    \[
    \begin{aligned}
    &\int_{\mathbb{R}^{2}} w_{k}^{2}|\mathcal{F}_{\alpha_{1},\alpha_{2}}\{f\}(\textbf{w})|^{2} d \textbf{w} \\
    =&\sin^{2}\alpha_{k}\int_{\mathbb{R}^{2}}\left( \frac{\partial}{\partial x_{k}} \rho(\textbf{x})\right) ^{2} d \textbf{x}+\cos^{2}\alpha_{k}\int_{\mathbb{R}^{2}} x_{k}^{2}\rho^{2}(\textbf{x})d\textbf{x} \\
     &+\sin^{2}\alpha_{k} \int_{\mathbb{R}^{2}} \rho^{2}(\textbf{x})\left|\left(\frac{\partial}{\partial x_{k}} e^{\textbf{u}(\textbf{x}) \theta(\textbf{x})}\right)\left(e^{-\textbf{u}(\textbf{x}) \theta(\textbf{x})}\right)\right|^{2} d \textbf{x}\\
    &  -2\sin\alpha_{k}\cos\alpha_{k}\int_{\mathbb{R}^{2}} x_{k}\rho^{2}(\textbf{x})Sc\left(   e_{k}  \left( \frac{\partial}{\partial x_{k}} e^{\textbf{u}(\textbf{x}) \theta(\textbf{x})} \right) \left( e^{-\textbf{u}(\textbf{x}) \theta(\textbf{x})} \right)\right)   d\textbf{x}.
    \end{aligned}
    \]
\end{thm}
    	Proof. Without loss of generality, we assume \(k=1\).
    	By Lemma 3.4, we can get the first equality, the second equality is obtained by (2.1),
    	\[
   \begin{aligned}
    	& \int_{\mathbb{R}^{2}} w_{1}^{2}|\mathcal{F}_{\alpha_{1},\alpha_{2}}\{f\}(\textbf{w})|^{2} d \textbf{w} \\
    	= &\int_{\mathbb{R}^{2}}\left| -i\sin\alpha_{1}\mathcal{F}_{\alpha_{1},\alpha_{2}}\left\{\frac{\partial f}{\partial x_{1}}\right\}(\textbf{w}) + \cos\alpha_{1}\mathcal{F}_{\alpha_{1},\alpha_{2}}\left\{x_{1}f\right\}(\textbf{w}) \right|^{2} d\textbf{w} \\
    	= &\int_{\mathbb{R}^{2}}\left( -i\sin\alpha_{1}\mathcal{F}_{\alpha_{1},\alpha_{2}}\left\{\frac{\partial f}{\partial x_{1}}\right\}(\textbf{w}) + \cos\alpha_{1}\mathcal{F}_{\alpha_{1},\alpha_{2}}\left\{x_{1}f\right\}(\textbf{w}) \right) \\
    	& \times \overline{\left( -i\sin\alpha_{1}\mathcal{F}_{\alpha_{1},\alpha_{2}}\left\{\frac{\partial f}{\partial x_{1}}\right\}(\textbf{w}) + \cos\alpha_{1}\mathcal{F}_{\alpha_{1},\alpha_{2}}\left\{x_{1}f\right\}(\textbf{w}) \right)} d\textbf{w} .
    	\end{aligned}
    	\]
    Expanding the parentheses and using lemma 4.4 in the second equality, the above equation reduces to
    		\[
    	\begin{aligned}
    	& \int_{\mathbb{R}^{2}} w_{1}^{2}|\mathcal{F}_{\alpha_{1},\alpha_{2}}\{f\}(\textbf{w})|^{2} d \textbf{w} \\
    	 =&\sin^{2}\alpha_{1}\int_{\mathbb{R}^{2}}\left|\mathcal{F}_{\alpha_{1},\alpha_{2}}\left\{\frac{\partial f}{\partial x_{1}}\right\}(\textbf{w}) \right| ^{2}d\textbf{w}+\cos^{2}\alpha_{1}\int_{\mathbb{R}^{2}}\left| \mathcal{F}_{\alpha_{1},\alpha_{2}}\left\{x_{1}f\right\}(\textbf{w})\right| ^{2}d\textbf{w}\\
    &  -\sin\alpha_{1}\cos\alpha_{1}\int_{\mathbb{R}^{2}}i\mathcal{F}_{\alpha_{1},\alpha_{2}}\left\{\frac{\partial f}{\partial x_{1}}\right\}(\textbf{w}) \overline{\mathcal{F}_{\alpha_{1},\alpha_{2}}\left\{x_{1}f\right\}(\textbf{w})} d\textbf{w} \\
    &  -\sin\alpha_{1}\cos\alpha_{1}\int_{\mathbb{R}^{2}}\mathcal{F}_{\alpha_{1},\alpha_{2}}\left\lbrace x_{1}f\right\rbrace (\textbf{w}) \overline{i\mathcal{F}_{\alpha_{1},\alpha_{2}}\left\{\frac{\partial f}{\partial x_{1}}\right\}(\textbf{w})} d\textbf{w}\\
    =&\sin^{2}\alpha_{1}\int_{\mathbb{R}^{2}}\left|\mathcal{F}_{\alpha_{1},\alpha_{2}}\left\{\frac{\partial f}{\partial x_{1}}\right\}(\textbf{w}) \right| ^{2}d\textbf{w}+\cos^{2}\alpha_{1}\int_{\mathbb{R}^{2}}\left| \mathcal{F}_{\alpha_{1},\alpha_{2}}\left\{x_{1}f\right\}(\textbf{w})\right| ^{2}d\textbf{w}
    \end{aligned}
    \]
    \[
    \begin{aligned}
    &-2\sin\alpha_{1}\cos\alpha_{1}\int_{\mathbb{R}^{2}}Sc\left( i\mathcal{F}_{\alpha_{1},\alpha_{2}}\left\{\frac{\partial f}{\partial x_{1}}\right\}(\textbf{w}) \overline{\mathcal{F}_{\alpha_{1},\alpha_{2}}\left\{x_{1}f\right\}(\textbf{w})}\right)  d\textbf{w}.
     \end{aligned}
     \]
       Using  Lemma 3.3, we have
       \[
     \begin{aligned}
   & \int_{\mathbb{R}^{2}} w_{1}^{2}|\mathcal{F}_{\alpha_{1},\alpha_{2}}\{f\}(\textbf{w})|^{2} d\textbf{w} \\
   =& \sin^{2}\alpha_{1}\int_{\mathbb{R}^{2}}\left|\frac{\partial f}{\partial x_{1}}\right| ^{2} d\textbf{x}+\cos^{2}\alpha_{1}\int_{\mathbb{R}^{2}}\left|x_{1}f\right| ^{2} d\textbf{x} \\
  & -2\sin\alpha_{1}\cos\alpha_{1} \int_{\mathbb{R}^{2}}Sc\left(  i\frac{\partial f}{\partial x_{1}}x_{1}\overline{f}\right) d\textbf{x}.
  \end{aligned}
  \]
  In the following calculations, by Lemma 4.2, we obtain the first equality.
  Using Lemma 4.3, we then obtain the second equality.
  \[
  \begin{aligned}
   & \int_{\mathbb{R}^{2}} w_{1}^{2}|\mathcal{F}_{\alpha_{1},\alpha_{2}}\{f\}(\textbf{w})|^{2} d\textbf{w} \\
  = &\sin^{2}\alpha_{1}\int_{\mathbb{R}^{2}}\left(  \frac{\partial}{\partial x_{1}} \rho(\textbf{x}) \right) ^{2}\!\!d\textbf{x}\! +\sin^{2}\alpha_{1}\!\int_{\mathbb{R}^{2}}\rho^{2}(\textbf{x}) \left|  \left( \frac{\partial}{\partial x_{1}} e^{\textbf{u}(\textbf{x}) \theta(\textbf{x})} \right)\!\! \left( e^{-\textbf{u}(\textbf{x}) \theta(\textbf{x})} \right)  \right|^{2}\!\!d\textbf{x} \\
  & +\cos^{2}\alpha_{1}\int_{\mathbb{R}^{2}}x_{1}^{2}\rho^{2}(\textbf{x})d\textbf{x}-2\sin\alpha_{1}\cos\alpha_{1}\int_{\mathbb{R}^{2}}Sc\left(  i\frac{\partial f}{\partial x_{1}}x_{1}\overline{f}\right)   d\textbf{x} \\
  =& \sin^{2}\alpha_{1}\int_{\mathbb{R}^{2}}\left( \frac{\partial}{\partial x_{1}} \rho(\textbf{x}) \right)^{2}\!\!d\textbf{x} +\sin^{2}\alpha_{1}\!\!\int_{\mathbb{R}^{2}}\!\!\rho^{2}(\textbf{x}) \left|  \left( \frac{\partial}{\partial x_{1}} e^{\textbf{u}(\textbf{x}) \theta(\textbf{x})} \right) \left( e^{-\textbf{u}(\textbf{x}) \theta(\textbf{x})} \right)   \right|^{2}\!\!d\textbf{x} \\
  & +\!\cos^{2}\alpha_{1}\!\!\int_{\mathbb{R}^{2}}\!\!x_{1}^{2}\rho^{2}(\textbf{x})d\textbf{x}\!\!-\!\!2\sin\alpha_{1}\cos\alpha_{1}\!\!\int_{\mathbb{R}^{2}} \!\!x_{1}\rho^{2}(\textbf{x})Sc\!\!\left( \!  i \! \left(\!\! \frac{\partial}{\partial x_{1}} e^{\textbf{u}(\textbf{x}) \theta(\textbf{x})} \!\!\right)\!\! \left( e^{-\textbf{u}(\textbf{x}) \theta(\textbf{x})}\! \right)\!\right)  \!\! d\textbf{x}.
  \end{aligned}
  \]
    Using similar calculations, the conclusion for k = 2 can be obtained.

   \hfill $\square$
\begin{rem}
	If $\alpha_1 = \alpha_2 = \dfrac{\pi}{2}$, The above theorem reduces to the result in \cite{12}
	\[
	\begin{aligned}
	&\int_{\mathbb{R}^{2}} \xi_{k}^{2}|F\{f\}(\underline{\xi})|^{2} d \underline{\xi}\\=  &\int_{\mathbb{R}^{2}}\left(\frac{\partial}{\partial x_{k}} \rho(\textbf{x})\right)^{2} d \textbf{x} +\int_{\mathbb{R}^{2}} \rho^{2}(\textbf{x})\left|NSc\left[\left(\frac{\partial}{\partial x_{k}} e^{\textbf{u}(\textbf{x}) \theta(\textbf{x})}\right)\left(e^{-\textbf{u}(\textbf{x}) \theta(\textbf{x})}\right)\right]\right|^{2} d \textbf{x} .
	\end{aligned}
	\]
In the complex case we have (\cite{3})
	\[
	\sigma_{\omega}^{2}=\int_{-\infty}^{\infty}w^{2}\left|F\{f\}(w)\right| ^{2}dw=\int_{-\infty}^{\infty} \rho^{\prime 2}(x) d x+\int_{-\infty}^{\infty} \rho^{2}(x) \theta^{\prime 2}(x) dx.
	\]
\end{rem}
Thus, our results constitute a natural generalization of the FT.
\quad

In Theorem 4.7, we present a one-directional uncertainty principle with respect to the spatial coordinate \(x_k\) and its Fourier counterpart \(w_k\).
\begin{thm}(UP in directional case)
	Let the quaternion-valued function \(f(\textbf{x})=\rho (\textbf{x}) e^{\textbf{u}(\textbf{x}) \theta(\textbf{x})}\). If \(f(\textbf{x}) \in L^{2}(\mathbb{R}^{2}; \mathbb{H})\), \(x_{k} f(\textbf{x})\), \(\frac{\partial}{\partial x_{k}} f(\textbf{x}) \in L^{2}(\mathbb{R}^{2}; \mathbb{H})\)  and \(\|f\|_{L^{2}}=1\), let \(e_{1}=i,e_{2}=j\), then for every \(k=1,2\), we have
\begin{equation}
	\begin{aligned}
	&\int_{\mathbb{R}^{2}} x_{k}^{2}|f(\textbf{x})|^{2} d \textbf{x}\int_{\mathbb{R}^{2}} w_{k}^{2}|\mathcal{F}_{\alpha_{1},\alpha_{2}}\{f\}(\textbf{w})|^{2} d \textbf{w} \\
	\geq &\frac{1}{4} \sin^{2}\alpha_{k}+ \sin^{2}\alpha_{k}COV_{x_{k} w_{k}}^{2}+ \cos^{2}\alpha_{k}\left( \int_{R^{2}}x_{k}^{2}\rho(\textbf{x})^{2}d\textbf{x}\right) ^{2}\\
	&  -\!\!2\sin\!\alpha_{k}\!\cos\alpha_{k}\!\!\int_{\mathbb{R}^{2}} \!\!x_{k}^{2}\rho^{2}(\!\textbf{x}) d \textbf{x} \int_{\mathbb{R}^{2}} \!\!x_{k}\rho^{2}(\!\textbf{x})Sc\!\left[e_{k} \!\left(\! \frac{\partial}{\partial x_{k}} e^{\textbf{u}(\textbf{x}) \theta(\textbf{x})} \right) \!\!\left( e^{-\textbf{u}(\textbf{x}) \theta(\textbf{x})} \right) \right]  \!\!d\textbf{x},
	\end{aligned}
\end{equation}
	where the absolute covariance is defined by
	\[
	COV_{x_{k}w _{k}}:=\int _{\mathbb{R}^{2}}\left| x_{k}\left(\frac{\partial}{\partial x_{k}} e^{\textbf{u}(\textbf{x}) \theta(\textbf{x})} \right) \!\!\left( e^{-\textbf{u}(\textbf{x}) \theta(\textbf{x})}\right)  \right| \rho ^{2}(\textbf{x})d\textbf{x}.
	\]
	The equality in (4.1) holds if and only if
	\[
		f(\textbf{x}) = \frac{1}{\sqrt[4]{\pi ^{2}\lambda_{1}\lambda_{2}\left| \sin \alpha_{1}\right| \left| \sin \alpha_{2}\right| }}e^{-\frac{1}{2\lambda_{1}}\frac{1}{\left| \sin \alpha_{1}\right| }x_{1}^2 -\frac{1}{2\lambda_{2}}\frac{1}{\left| \sin \alpha_{2}\right| }x_{2}^2}e^{\textbf{u}(\textbf{x})\theta(\textbf{x})}
		\] and \( \left( \frac{\partial}{\partial x_k} e^{\textbf{u}(\textbf{x}) \theta(\textbf{x})} \right) e^{-\textbf{u}(\textbf{x}) \theta(\textbf{x})} = \beta_kx_k \). Here, \(\lambda_{1},\lambda_{2} > 0,\, \beta_1,\beta_2\) is pure quaternions.
	\end{thm}
Proof. By Theorem 4.5, we have
 \[
\begin{aligned}
&\int_{\mathbb{R}^{2}} x_{k}^{2}|f(\textbf{x})|^{2} d \textbf{x}\int_{\mathbb{R}^{2}} w_{k}^{2}|\mathcal{F}_{\alpha_{1},\alpha_{2}}\{f\}(\textbf{w})|^{2} d \textbf{w} \\
	 =&\int_{\mathbb{R}^{2}} x_{k}^{2}\rho^{2}(\textbf{x}) d \textbf{x} \int_{\mathbb{R}^{2}}\sin^{2}\alpha_{k}\left( \frac{\partial}{\partial x_{k}} \rho(\textbf{x})\right) ^{2} d \textbf{x}+\cos^{2}\alpha_{k}\left(\int_{\mathbb{R}^{2}} x_{k}^{2}\rho^{2}(\textbf{x}) d \textbf{x}\right)^{2} \\
	& +\sin^{2}\alpha_{k} \int_{\mathbb{R}^{2}} x_{k}^{2}\rho^{2}(\textbf{x}) d \textbf{x} \int_{\mathbb{R}^{2}} \rho^{2}(\textbf{x})\left|\left(\frac{\partial}{\partial x_{k}} e^{\textbf{u}(\textbf{x}) \theta(\textbf{x})}\right)\left(e^{-\textbf{u}(\textbf{x}) \theta(\textbf{x})}\right)\right|^{2} d \textbf{x}
	\end{aligned}
	\]
	\begin{equation}
	\begin{aligned}
	&  -\!\!2\sin\alpha_{k}\!\cos\alpha_{k}\!\int_{\mathbb{R}^{2}} \!\!x_{k}^{2}\rho^{2}(\textbf{x})d \textbf{x}\! \int_{\mathbb{R}^{2}}\!\! x_{k}\rho^{2}(\textbf{x})Sc\!\left[ e_{k} \!\!\left( \!\frac{\partial}{\partial x_{k}} e^{\textbf{u}(\textbf{x}) \theta(\textbf{x})} \right)\!\! \left( e^{-\textbf{u}(\textbf{x}) \theta(\textbf{x})} \right) \right]\!\! d\textbf{x}\!.
	\end{aligned}
	\end{equation}
For the first term in (4.2), we apply H\(\ddot{o}\)lder's inequality, then
\begin{equation}
\begin{aligned}
& \int_{\mathbb{R}^{2}} x_{k}^{2} \rho^{2}(\textbf{x}) d \textbf{x}  \int_{\mathbb{R}^{2}}\sin^{2}\alpha_{k} \left(  \frac{\partial}{\partial x_{k}} \rho(\textbf{x}) \right) ^{2} d \textbf{x} \\
\geq &\left( \int_{\mathbb{R}^{2}} \left| x_{k} \rho (\textbf{x})\sin\alpha_{k}\left(  \frac{\partial}{\partial x_{k}} \rho(\textbf{x}) \right) \right| d \textbf{x} \right)^{2}\\
\geq &\left| \int_{\mathbb{R}^{2}} x_{k} \rho(\textbf{x})\sin\alpha_{k} \left( \frac{\partial}{\partial x_{k}} \rho(\textbf{x}) \right)  d \textbf{x} \right|^{2}.
\end{aligned}
\end{equation}
 In the following expression, note that \(x_{1}\rho^{2}(\textbf{x})\left.\right|_{-\infty}^{+\infty} \) is zero and the second term equals 1 since we assume the signal is unit energy (i.e. \( \|f\|_{L^2} = 1 \)), we have
\[
\begin{aligned}
&\int_{\mathbb{R}^{2}} x_{1} \rho(\textbf{x}) \sin\alpha_{1}\left(  \frac{\partial}{\partial x_{1}} \rho(\textbf{x}) \right)  d \textbf{x}\\
=&\sin\alpha_{1} \int_{\mathbb{R}} \left( x_{1}\rho^{2}(\textbf{x})\left.\right|_{-\infty}^{+\infty}-\int_{\mathbb{R}}\rho^{2}(\textbf{x})dx_{1}-\int_{\mathbb{R}}x_{1}\rho(\textbf{x})\frac{\partial\rho(\textbf{x})}{\partial x_{1}}dx_{1}\right) dx_{2} \\
=&-\sin\alpha_{1}-\int_{\mathbb{R}^{2}} x_{1} \rho(\textbf{x}) \sin\alpha_{1}\left(  \frac{\partial}{\partial x_{1}} \rho(\textbf{x}) \right)  d \textbf{x}\\
=& -\frac{1}{2}\sin\alpha_{1}.
	\end{aligned}
\]
 Using similar calculations, we have
	\begin{equation}
	\int_{\mathbb{R}^{2}} x_{k} \rho(\textbf{x}) \sin\alpha_{k}\left(  \frac{\partial}{\partial x_{k}} \rho(\textbf{x}) \right)  d \textbf{x}= -\frac{1}{2}\sin\alpha_{k}.
	\end{equation}
	Substituting (4.4) into (4.3), we obtain the following equation.
	\begin{equation}
	 \int_{\mathbb{R}^{2}} x_{k}^{2} \rho^{2}(\textbf{x}) d \textbf{x}  \int_{\mathbb{R}^{2}} \sin^{2}\alpha_{k} \left(  \frac{\partial}{\partial x_{k}} \rho(\textbf{x}) \right) ^{2} d \textbf{x} \geq\frac{1}{4}\sin^{2}\alpha_{k}.
	\end{equation}
	Similarly, For the third term in (4.2), we have
\begin{equation}
	\begin{aligned}
	&  \int_{\mathbb{R}^{2}} x_{k}^{2} \rho^{2}(\textbf{x}) d \textbf{x}  \int_{\mathbb{R}^{2}} \rho^{2}(\textbf{x}) \left|  \left( \frac{\partial}{\partial x_{k}} e^{\textbf{u}(\textbf{x}) \theta(\textbf{x})} \right) \left( e^{-\textbf{u}(\textbf{x}) \theta(\textbf{x})} \right)  \right|^{2} d \textbf{x}  \\
	\geq& \left( \int_{\mathbb{R}^{2}} \left| x_{k} \left( \frac{\partial}{\partial x_{k}} e^{\textbf{u}(\textbf{x}) \theta(\textbf{x})} \right) \left( e^{-\textbf{u}(\textbf{x}) \theta(\textbf{x})} \right) \right| \rho^{2}(\textbf{x}) d \textbf{x} \right)^{2}= COV_{x_{k} w_{k}}^{2} .
	\end{aligned}
	\end{equation}
	Connecting (4.2), (4.5) and (4.6), the inequality (4.1) holds.
	
	Now, we describe the conditions under which equality holds in  (4.1). The first equality in (4.3) holds if and only if there exists a constant \( \lambda_{k}>0 \) such that
	\[
	|x_{k}\rho| = \lambda_{k} \left|\sin \alpha_{k}\frac{\partial}{\partial x_{k}}\rho \right|.
	\]
	 The second equality in (4.3) holds if and only if  \(x_{k} \) and \(\sin \alpha_{k}\frac{\partial}{\partial x_{k}}\rho\)  have the same sign or opposite signs.
	 If  \(x_{k} \) and \(\sin \alpha_{k}\frac{\partial}{\partial x_{k}}\rho\) possess the same sign, then \(x_{k}\rho = \lambda_{k} \left| \sin \alpha_{k}\right| \frac{\partial}{\partial x_{k}}\rho \), that is, \( \frac{\frac{\partial}{\partial x_{k}}\rho}{\rho} = \frac{1}{\lambda_{k}}\frac{1}{\left| \sin \alpha_{k}\right| } x_{k} \). Taking the indefinite integral of both sides, we have
	 \[
	 \rho(\textbf{x}) = e^{\frac{1}{2\lambda_{k}}\frac{1}{\left| \sin \alpha_{k}\right| }x_{k}^2 + c}.
	 \]
	 But it is  known the function \( \rho(\textbf{x}) = e^{\frac{1}{2\lambda_{k}}{\left| \sin \alpha_{k}\right| }x_{k}^2 + c}\) cannot be in  the space \(L^2(\mathbb{R}^{2}) \). Therefore, \(x_{k} \) and \({\sin \alpha_{k}}\frac{\partial}{\partial x_{k}}\rho\)  must have opposite signs. It means that  \( \frac{\frac{\partial}{\partial x_{k}}\rho}{\rho} = -\frac{1}{\lambda_{k}}\frac{1}{\left| \sin \alpha_{k}\right| }x_{k}\),  then we have
	 \[
	\rho(\textbf{x}) = e^{-\frac{1}{2\lambda_{k}}\frac{1}{\left| \sin \alpha_{k}\right| }x_{k}^2 + c}.
	 \]
	 Clearly, the equation holds in (4.6) if and only if there exists a number \(\beta_{k}  \) such that
	 \[
	 \beta_{k}x_{k}\rho(\textbf{x})=\rho(\textbf{x}) \left( \frac{\partial}{\partial x_{k}} e^{\textbf{u}(\textbf{x}) \theta(\textbf{x})} \right) \left( e^{-\textbf{u}(\textbf{x}) \theta(\textbf{x})} \right)
	 \]
By Lemma 4.1, Sc\(\left[\left( \frac{\partial}{\partial x_{k}} e^{\textbf{u}(\textbf{x}) \theta(\textbf{x})} \right) \left( e^{-\textbf{u}(\textbf{x}) \theta(\textbf{x})} \right) \right]=0\), then \(\beta_{k}\) must be a pure quaternion. We have
	 \[
	 \begin{aligned}
	\left( \frac{\partial}{\partial x_{k}} e^{\textbf{u}(\textbf{x}) \theta(\textbf{x})} \right) \left( e^{-\textbf{u}(\textbf{x}) \theta(\textbf{x})} \right)
	 & = \beta_{k} x_{k}.
	 \end{aligned}
	 \]
	\(\beta_{1},\beta_{2}\) is pure quaternions.
	
		Since signals we discuss are of unit energy, that is
	\[
	\begin{aligned}
	&\int_{\mathbb{R}^{2}}f(\textbf{x})^{2}d\textbf{x}=\int_{\mathbb{R}^{2}}\rho^{2} (\textbf{x})d\textbf{x}=1\\
	&=e^{2c}\int_{\mathbb{R}^{2}}e^{-\frac{1}{\lambda_{1}}\frac{1}{\left| \sin \alpha_{1}\right| }x_{1}^{2}-\frac{1}{\lambda_{2}}\frac{1}{\left| \sin \alpha_{2}\right| }x_{2}^{2} }d\textbf{x}\\
	&=e^{2c}\int_{\mathbb{R}}\int_{\mathbb{R}}e^{-\frac{1}{\lambda_{1}}\frac{1}{\left| \sin \alpha_{1}\right| }x_{1}^{2}}dx_{1}e^{-\frac{1}{\lambda_{2}}\frac{1}{\left| \sin \alpha_{2}\right| }x_{2}^{2}}dx_{2}\\
	&=e^{2c}\pi\lambda_{1}^{\frac{1}{2}}\lambda_{2}^{\frac{1}{2}}\left|  \sin\alpha_1\right| ^{\frac{1}{2}}\left| \sin\alpha_2\right| ^{\frac{1}{2}}.
	\end{aligned}
	\]
	We derive that \(\lambda\) and c should satisfy \(e^{4c}\pi^{2}\lambda_{1}\lambda_{2} \left| \sin\alpha_1\right| \left| \sin\alpha_2 \right| =1\).
	
	 This completes the proof.

\hfill $\square$
\begin{rem}
	If we work in the space we defined in Definition 2.1, we obtain a stronger result concerning \(\beta\). i.e. \(\beta_{1}\beta_{2}=\beta_{2}\beta_{1}\).
	\[
	\frac{\partial}{\partial x_{1}\partial x_{2}} e^{\textbf{u}(\textbf{x}) \theta(\textbf{x})}=\beta_{1}x_{1}\frac{\partial}{\partial x_{2}} e^{\textbf{u}(\textbf{x}) \theta(\textbf{x})},\frac{\partial}{\partial x_{2}\partial x_{1}} e^{\textbf{u}(\textbf{x}) \theta(\textbf{x})}=\beta_{2}x_{2}\frac{\partial}{\partial x_{1}} e^{\textbf{u}(\textbf{x}) \theta(\textbf{x})}.
	\]
	Since \(e^{\textbf{u}(\textbf{x}) \theta(\textbf{x})} \in S\left(\mathbb{R}^{2} ; \mathbb{H}\right)\) and Theroem 4.7 shows \(\frac{\partial}{\partial x_{k}} e^{\textbf{u}(\textbf{x}) \theta(\textbf{x})}=\beta_{k}x_{k}e^{\textbf{u}(\textbf{x}) \theta(\textbf{x})}\), we get \(\frac{\partial}{\partial x_{1}\partial x_{2}} e^{\textbf{u}(\textbf{x}) \theta(\textbf{x})}=\frac{\partial}{\partial x_{2}\partial x_{1}} e^{\textbf{u}(\textbf{x}) \theta(\textbf{x})}\), i.e.
	\[
	\beta_{1}x_{1}\frac{\partial}{\partial x_{2}} e^{\textbf{u}(\textbf{x}) \theta(\textbf{x})}=\beta_{2}x_{2}\frac{\partial}{\partial x_{1}} e^{\textbf{u}(\textbf{x}) \theta(\textbf{x})},
	\]
	\[
		\beta_{1}x_{1}\beta_{2}x_{2}e^{\textbf{u}(\textbf{x}) \theta(\textbf{x})}=\beta_{2}x_{2}\beta_{1}x_{1}e^{\textbf{u}(\textbf{x})\theta(\textbf{x})}.
	\]
	Finally we have \(\beta_{1}\beta_{2}=\beta_{2}\beta_{1}\).
	
\end{rem}

	 By introducing the absolute covariance term \(	COV_{x_{k}w _{k}}\), Theorem 4.7 presents a more refined spatial-domain UP. Moreover, since the standard covariance satisfies \(\left|Cov_{x_{k}w _{k}}\right|\leq COV_{x_{k}w _{k}} \), the inequality can be appropriately relaxed. Based on this relation, Theorem 4.7 directly leads to a more classical form of the UP, namely Corollary 4.9 below.
	 \begin{cor}
	 	Let \(f(\textbf{x})=|f(\textbf{x})| e^{\textbf{u}(\textbf{x}) \theta(\textbf{x})}\). If \(f(\textbf{x}) \in L^{2}(\mathbb{R}^{2}, \mathbb{H})\), \(x_{k} f(\textbf{x})\), \(\frac{\partial}{\partial x_{k}} f(\textbf{x}) \in L^{2}(\mathbb{R}^{2}, \mathbb{H})\)  and \(\|f\|_{L^{2}}=1\), then
	 	\[
	 	\begin{aligned}
	 	&\int_{\mathbb{R}^{2}} x_{k}^{2}|f(\textbf{x})|^{2} d \textbf{x}\int_{\mathbb{R}^{2}} w_{k}^{2}|\mathcal{F}_{\alpha_{1},\alpha_{2}}\{f\}(\textbf{w})|^{2} d \textbf{w} \\
	 	\geq& \frac{1}{4} \sin^{2}\alpha_{k}+ \sin^{2}\alpha_{k}\left| Cov_{x_{k} w_{k}}\right| ^{2}+ \cos^{2}\alpha_{k}\left( \int_{\mathbb{R}^{2}}x_{k}^{2}\rho^{2}(\textbf{x})d\textbf{x}\right) ^{2}\\
	 	&  -\!\!2\!\sin\alpha_{k}\!\cos\alpha_{k}\!\int_{\mathbb{R}^{2}}\!\! x_{k}^{2}\rho^{2}(\textbf{x}) d \textbf{x}\!\! \int_{\mathbb{R}^{2}}\!\! x_{k}\rho^{2}(\textbf{x})Sc \! \left[ e_{k}\!\left( \!\frac{\partial}{\partial x_{k}} e^{\textbf{u}(\textbf{x}) \theta(\textbf{x})} \right)\!\! \left( e^{-\textbf{u}(\textbf{x}) \theta(\textbf{x})} \right) \right]\!\! d\textbf{x},
	 	\end{aligned}
	 	\]
	 	where the absolute covariance is defined by
	 	\[
	 	Cov_{x_{k}w _{k}}:=\int _{\mathbb{R}^{2}} x_{k} \left(\frac{\partial}{\partial x_{k}} e^{\textbf{u}(\textbf{x}) \theta(\textbf{x})} \right) \!\!\left( e^{-\textbf{u}(\textbf{x}) \theta(\textbf{x})}\right)  \rho ^{2}(\textbf{x})d\textbf{x}.
	 	\]
	 \end{cor}

 Now, we proceed to prove the main theorem.

\textbf{Proof of Theorem 1.1.} By Theorem 4.5, we have
 	\[
 	\begin{aligned}
 		&\int_{\mathbb{R}^{2}}\left| \textbf{x}\right| ^{2} |f(\textbf{x})|^{2} d \textbf{x}\int_{\mathbb{R}^{2}}\left| \textbf{w}\right| ^{2} |\mathcal{F}_{\alpha_{1},\alpha_{2}}\{f\}(\textbf{w})|^{2} d \textbf{w} \\
 		=& \int_{\mathbb{R}^{2}} \sum_{k=1}^{2} x_{k}^{2} \rho^{2} d\textbf{x}  \int_{\mathbb{R}^{2}} \sum_{k=1}^{2} w_{k}^{2} |\mathcal{F}_{\alpha_{1},\alpha_{2}}\{f\}(\textbf{w})|^{2} d\textbf{w} \\
 		=& \int_{\mathbb{R}^{2}} \sum_{k=1}^{2} x_{k}^{2} \rho^{2} d\textbf{x}  \int_{\mathbb{R}^{2}} \sum_{k=1}^{2}sin^{2} \alpha_k  \left( \frac{\partial}{\partial x_{k}} \rho(\textbf{x})\right)  ^{2} d\textbf{x} \\ &+\int_{\mathbb{R}^{2}} \sum_{k=1}^{2} x_{k}^{2} \rho^{2} d\textbf{x} \int_{\mathbb{R}^{2}} \sum_{k=1}^{2}\cos^{2}\alpha_{k} x_{k}^{2} \rho^{2} d\textbf{x} \\
 		&+ \int_{\mathbb{R}^{2}} \sum_{k=1}^{2} x_{k}^{2} \rho^{2} d\textbf{x}  \int_{\mathbb{R}^{2}} \sum_{k=1}^{2}sin^{2} \alpha_k \rho^{2}\left|  \left( \frac {\partial }{\partial x_{k}}e^{\underline {{u}}(\underline {{x}})\theta (\underline {{x}})}\right) e^{-\underline {{u}}(\underline {{x}})\theta (\underline {{x}})}\right| ^{2}d\textbf{x}
 		\end{aligned}
 		\]
 		\begin{equation}
 		\begin{aligned}
 		&-\!\!\int_{\mathbb{R}^{2}} \!\sum_{k=1}^{2} x_{k}^{2} \rho^{2} d\textbf{x}\! \int_{\mathbb{R}^{2}}\!\sum_{k=1}^{2}\!2\sin\alpha_{k}\cos\alpha_{k}\!x_{k}\rho^{2}Sc\! \!\left[ e_{k}\!\left( \frac{\partial}{\partial x_{k}} e^{\textbf{u}(\textbf{x}) \theta(\textbf{x})} \!\!\right) \!\!\left( e^{-\textbf{u}(\textbf{x}) \theta(\textbf{x})} \right) \!\right]\!\! d\textbf{x}.
 		\end{aligned}
 		\end{equation}
 		The first inequality applies the continuous case of the Cauchy-Schwarz inequality, while the second employs its discrete case, the last equality follows from (4.4), we get
 		
 		\[
 		\begin{aligned}
 		& \int_{\mathbb{R}^{2}} \sum_{k=1}^{2} x_{k}^{2} \rho^{2} d\textbf{x} \int_{\mathbb{R}^{2}} \sum_{k=1}^{2}sin^{2} \alpha_k  \left( \frac{\partial}{\partial x_{k}} \rho(\textbf{x})\right)  ^{2} d\textbf{x}\\
 		\geq& \left| \int_{\mathbb{R}^{2}}\left( \sum_{k=1}^{2} x_{k}^{2} \rho^{2}\right) ^{\frac{1}{2}}\left( \sum_{k=1}^{2}sin^{2} \alpha_k  \left( \frac{\partial}{\partial x_{k}} \rho(\textbf{x})\right)  ^{2}\right) ^{\frac{1}{2}}d\textbf{x}\right| ^{2}
 			\end{aligned}
 		\]
 		\begin{equation}
 		\geq\left| \int_{\mathbb{R}^{2}}\sum_{k=1}^{2}x_{k}\rho \sin\alpha_k\frac{\partial}{\partial x_{k}} \rho(\textbf{x})d\textbf{x}\right| ^{2}\\
 		=\frac{1}{4}\left| \sin\alpha_1+\sin\alpha_2\right| ^{2}.
 		\end{equation}
 	Similarly, we have
 	\begin{equation}
 	\begin{aligned}
 	& \int_{\mathbb{R}^{2}} \sum_{k=1}^{2} x_{k}^{2} \rho^{2} d\textbf{x} \int_{\mathbb{R}^{2}} \sum_{k=1}^{2}sin^{2} \alpha_k \rho^{2}\left|\left(  \frac{\partial}{\partial x_{k}} e^{\textbf{u}(\textbf{x}) \theta(\textbf{x})} \right) \!\!\left( e^{-\textbf{u}(\textbf{x}) \theta(\textbf{x})}\right) \right| ^{2}d\textbf{x} \\
 	\geq&\left( \int_{\mathbb{R}^{2}} \sum_{k=1}^{2}\left| x_{k}\rho ^{2}\sin\alpha_k  \left(  \frac{\partial}{\partial x_{k}} e^{\textbf{u}(\textbf{x}) \theta(\textbf{x})} \right) \!\!\left( e^{-\textbf{u}(\textbf{x}) \theta(\textbf{x})}\right)\right| d\textbf{x}\right) ^{2}.
 	\end{aligned}
 	\end{equation}
 	
 	Connecting (4.7), (4.8) and (4.9), the inequality (1.5) holds.
 	
 	Similarly, as in Theorem 4.7, the equality (1.5) holds if and only if there exists a positive number \( \lambda,\beta  \) such that
 	\[
 	\left| x_{k}\rho \right| =\lambda\left| \sin \alpha_{k}\frac {\partial\rho }{\partial x_{k}}\right|,\,\beta x_{k}\rho(\textbf{x})=\rho(\textbf{x}) \left( \frac{\partial}{\partial x_{k}} e^{\textbf{u}(\textbf{x}) \theta(\textbf{x})} \right) \left( e^{-\textbf{u}(\textbf{x}) \theta(\textbf{x})} \right), \]
 	we have \(\rho(\textbf{x}) = e^{-\frac{1}{2\lambda}\left( \frac{1}{\left| \sin\alpha_1\right| }x_{1}^2 +\frac{1}{\left| \sin\alpha_2\right| }x_{2}^2\right) + c}\) and \( \left( \frac{\partial}{\partial x_k} e^{\textbf{u}(\textbf{x}) \theta(\textbf{x})} \right) e^{-\textbf{u}(\textbf{x}) \theta(\textbf{x})} = \beta x_k \). Here, \(\lambda >0\) and \(\beta\) is pure quaternions.
 	
 	Since signals we discuss are of unit energy, that is
 	\[
 	\begin{aligned}
 	&\int_{\mathbb{R}^{2}}f(\textbf{x})^{2}d\textbf{x}=\int_{\mathbb{R}^{2}}\rho (\textbf{x})^{2}d\textbf{x}=1\\
 	=&e^{2c}\int_{\mathbb{R}^{2}}e^{-\frac{1}{\lambda}\frac{1}{\left| \sin \alpha_{1}\right| }x_{1}^{2}-\frac{1}{\lambda}\frac{1}{\left| \sin \alpha_{2}\right| }x_{2}^{2} }d\textbf{x}=e^{2c}\int_{\mathbb{R}}\int_{\mathbb{R}}e^{-\frac{1}{\lambda}\frac{1}{\left| \sin \alpha_{1}\right| }x_{1}^{2}}dx_{1}e^{-\frac{1}{\lambda}\frac{1}{\left| \sin \alpha_{2}\right| }x_{2}^{2}}dx_{2}\\
 	=&e^{2c}\pi\lambda \left| \sin\alpha_1\right| ^{\frac{1}{2} }\left| \sin\alpha_2\right| ^{\frac{1}{2} },
 	\end{aligned}
 	\]
 	 we derive that \(\lambda\) and c should satisfy \(e^{2c}\pi\lambda\left| \sin\alpha_1\right| ^{\frac{1}{2} }\left| \sin\alpha_2\right| ^{\frac{1}{2} } =1\).
 	
 	Now we prove the final condition for the equality to hold. We consider a quaternionic valued signal
 	\[
 	f(x)=e^{-\frac{1}{2\lambda}\left( \frac{1}{\left| \sin\alpha_1\right| }x_{1}^2 +\frac{1}{\left| \sin\alpha_2\right| }x_{2}^2\right) + c}e^{\textbf{u}(\textbf{x})\theta(\textbf{x})},
 	\]
 	where \( \left( \frac{\partial}{\partial x_k} e^{\textbf{u}(\textbf{x}) \theta(\textbf{x})} \right) e^{-\textbf{u}(\textbf{x}) \theta(\textbf{x})} = \beta x_k \), \(\lambda>0 \), $\alpha_1 ,\alpha _2 \neq n\pi $ and \(\beta\)  is pure quaternions.  Without loss of generality, we consider the case where \(\sin\alpha_{1}>0, \sin\alpha_{2}> 0\), the other cases can be obtained similarly by taking the absolute value.
 	
 	Computing directly, we have
 	\begin{equation}
 	\begin{aligned}
 	\int_{\mathbb{R}^{2}}\left| \textbf{x}\right| ^{2} |f(\textbf{x})|^{2} d \textbf{x}
 	&=\int_{\mathbb{R}^{2}}e^{-\frac{1}{\lambda}\left( \frac{1}{\sin\alpha_1}x_{1}^2 +\frac{1}{\sin\alpha_2}x_{2}^2\right) + 2c}\left( x_{1}^{2}+x_{2}^{2} \right) d\textbf{x}\\
 	&=\frac{\pi}{2}\lambda^{2}e^{2c}\sin\alpha_1^{\frac{1}{2}}\sin\alpha_2^{\frac{1}{2}}\left(\sin\alpha_1+\sin\alpha_2\right).
 		\end{aligned}
 	\end{equation}
 	Since \(\rho (\textbf{x})=e^{-\frac{1}{2\lambda}\left( \frac{1}{\sin\alpha_1}x_{1}^2 +\frac{1}{\sin\alpha_2}x_{2}^2\right) + c}, \left( \frac{\partial}{\partial x_k} e^{\textbf{u}(\textbf{x}) \theta(\textbf{x})} \right) e^{-\textbf{u}(\textbf{x}) \theta(\textbf{x})} = \beta x_k\), and by Theorem 4.5, we have
 	 	\begin{equation}
 	\begin{aligned}
 	&\int_{\mathbb{R}^{2}}\left| \textbf{w}\right| ^{2} |\mathcal{F}_{\alpha_{1},\alpha_{2}}\{f\}(\textbf{w})|^{2} d \textbf{w}\\
 	=& \int_{\mathbb{R}^{2}} \sum_{k=1}^{2}sin^{2} \alpha_k  \left( \frac{\partial}{\partial x_{k}} \rho(\textbf{x})\right)  ^{2} d\textbf{x}+\int_{\mathbb{R}^{2}} \sum_{k=1}^{2}\cos^{2}\alpha_{k} x_{k}^{2} \rho^{2} d\textbf{x}  \\
 	& + \int_{\mathbb{R}^{2}} \sum_{k=1}^{2}sin^{2} \alpha_k \rho^{2}\left| \left(  \frac{\partial}{\partial x_{k}} e^{\textbf{u}(\textbf{x}) \theta(\textbf{x})} \right) \!\!\left( e^{-\textbf{u}(\textbf{x}) \theta(\textbf{x})}\right)\right| ^{2}d\textbf{x}\\
 	& -\int_{\mathbb{R}^{2}}\sum_{k=1}^{2}2\sin\alpha_{k}\cos\alpha_{k}x_{k}\rho^{2}(\textbf{x})Sc \left[ e_{k}\left( \frac{\partial}{\partial x_{k}} e^{\textbf{u}(\textbf{x}) \theta(\textbf{x})} \right) \left( e^{-\textbf{u}(\textbf{x}) \theta(\textbf{x})} \right) \right] d\textbf{x}\\
 	=&\frac{\pi}{2}e^{2c}\sin\alpha_1^{\frac{1}{2}}\sin\alpha_2^{\frac{1}{2}}\left(\sin\alpha_1+\sin\alpha_2\right)\\ &+\frac{\pi}{2}e^{2c}\lambda^{2}\sin\alpha_1^{\frac{1}{2}}\sin\alpha_2^{\frac{1}{2}}\left(\cos ^{2}\alpha_{1}\sin\alpha_1+\cos ^{2}\alpha_{2}\sin\alpha_2\right) \\
 	&+ \frac{\pi}{2}\lambda^{2}e^{2c}\sin\alpha_1^{\frac{1}{2}}\sin\alpha_1^{\frac{1}{2}}\left| \beta\right| ^{2} \left(\sin^{3}\alpha_1+\sin^{3}\alpha_2\right)  \\
 	&-\pi\lambda^{2}e^{2c}\sin\alpha_1^{\frac{1}{2}}\sin\alpha_1^{\frac{1}{2}}\left( \sin^{2}\alpha_{1}\cos\alpha_{1}Sc \left[ e_{1}\beta \right]+\sin^{2}\alpha_{2}\cos\alpha_{2}Sc \left[ e_{2}\beta \right]\right).
    \end{aligned}
 	 	\end{equation}
 	Using (4.10), (4.11) and \(e^{2c}\pi\lambda\sin\alpha_1^{\frac{1}{2} }\sin\alpha_2^{\frac{1}{2} } =1\),  we have
 		\[
 	\begin{aligned}
 	&\int_{\mathbb{R}^{2}}\left| \textbf{x}\right| ^{2} |f(\textbf{x})|^{2} d \textbf{x}\int_{\mathbb{R}^{2}}\left| \textbf{w}\right| ^{2} |\mathcal{F}_{\alpha_{1},\alpha_{2}}\{f\}(\textbf{w})|^{2} d \textbf{w} \\
 	=&\frac {1}{4}\left( \sin\alpha_1+\sin\alpha_2\right) ^{2} +\frac {1}{4}\lambda^{2}\left( \sin\alpha_1+\sin\alpha_2\right)\left( \cos^{2} \alpha_{1}\sin\alpha_1+\cos^{2} \alpha_{2}\sin\alpha_2\right)\\
 	&+\frac {1}{4}\lambda^{2}\left| \beta\right| ^{2}\left( \sin^{4}\alpha_1+ sin\alpha_1sin^{3}\alpha_2+sin^{3}\alpha_1sin\alpha_2+\sin^{4}\alpha_2\right)
 	\end{aligned}
 	\]
 	\begin{equation}
 	\begin{aligned}
 	&-\frac {1}{2}\lambda^{2}\left( \sin\alpha_1+\sin\alpha_2\right) \left( \sin^{2}\alpha_{1}\cos\alpha_{1}Sc \left[ e_{1}\beta \right]+\sin^{2}\alpha_{2}\cos\alpha_{2}Sc \left[ e_{2}\beta \right]\right).
 	\end{aligned}
 	\end{equation}
 	Thus, we have computed the left-hand side of (1.5). Next, we proceed to prove the right-hand side.
 	It is easy to know that
 	\[
 	\begin{aligned}
 	 COV_{\textbf{x}\textbf{w}}&=\sum_{k=1}^{2}\int _{\mathbb{R}^{2}}\left| x_{k}\sin\alpha_{k} \left(  \frac{\partial}{\partial x_{k}} e^{\textbf{u}(\textbf{x}) \theta(\textbf{x})} \right) \!\!\left( e^{-\textbf{u}(\textbf{x}) \theta(\textbf{x})}\right) \right| \rho ^{2}(\textbf{x})d\textbf{x}\\
 	 &=\frac{\pi}{2}\lambda^{2} e^{2c}\sin\alpha_1^{\frac{1}{2}}\sin\alpha_1^{\frac{1}{2}}\left| \beta\right| \left( \sin^{2}\alpha_1+ \sin^{2}\alpha_2\right),
 	  	\end{aligned}
 	 \]
 	  then
 	\begin{equation}
 	\begin{aligned}
 	&COV_{\textbf{x}\textbf{w}}^{2}\\
 	=&e^{4c}\frac {\pi ^{2}}{4}\lambda^{4}\sin \alpha_1\sin \alpha_2\left| \underline{\beta}\right|^{2}\left(\sin ^{4}\alpha_1+  2\sin^{2}\alpha_1\sin^{2}\alpha_2+sin^{4}\alpha_2\right)\\
 	=&\frac {1}{4}\lambda^{2}\left| \beta\right|^{2}\left(\sin ^{4}\alpha_1+  2\sin^{2}\alpha_1\sin^{2}\alpha_2+sin^{4}\alpha_2\right).
 	 	\end{aligned}
 	\end{equation}
 By direct computation, we obtain the third and fourth terms on the right-hand side of Theorem 1.1.
 	\[
 	\begin{aligned}	
 	&\int_{\mathbb{R}^{2}} \sum_{k=1}^{2} x_{k}^{2} \rho^{2} d\textbf{x} \int_{\mathbb{R}^{2}} \sum_{k=1}^{2}\cos^{2}\alpha_{k} x_{k}^{2} \rho^{2} d\textbf{x} \\
 	=&\frac {1}{4}\lambda^{2}\left( \sin\alpha_1+\sin\alpha_2\right)\left( \cos^{2} \alpha_{1}\sin\alpha_1+\cos^{2} \alpha_{2}\sin\alpha_2\right),\\
 	&\!\!\int_{\mathbb{R}^{2}} \!\!\sum_{k=1}^{2} x_{k}^{2} \rho^{2} d\textbf{x}\!\!  \int_{\mathbb{R}^{2}}\!\sum_{k=1}^{2}2\sin\alpha_{k}\cos\alpha_{k}x_{k}\rho^{2}(\textbf{x})Sc\!\! \left[ \!e_{k}\!\left( \frac{\partial}{\partial x_{k}} e^{\textbf{u}(\textbf{x}) \theta(\textbf{x})} \!\!\right)\!\! \left(\! e^{-\textbf{u}(\textbf{x}) \theta(\textbf{x})} \right)\!\right]\!\! d\textbf{x}. \\
 	=&\frac {1}{2}\lambda^{2}\left( \sin\alpha_1+\sin\alpha_2\right) \left( \sin^{2}\alpha_{1}\cos\alpha_{1}Sc \left[ e_{1}\underline{\beta} \right]+\sin^{2}\alpha_{2}\cos\alpha_{2}Sc \left[ e_{2}\underline{\beta} \right]\right).
 	\end{aligned}
 	\]
 	Thus, we have computed every term on the right-hand side of (1.5).
 By comparing the left-hand side and the right-hand side of (1.5), to ensure that equality holds, we require the following condition:
 	\[
 	COV_{\textbf{x}\textbf{w}}^{2}=\frac {1}{4}\lambda^{2}\left| \underline{\beta}\right| ^{2}\left( \sin^{4}\alpha_1+ sin\alpha_1sin^{3}\alpha_2+sin^{3}\alpha_1sin\alpha_2+\sin^{4}\alpha_2\right).
 	\]
 From (4.12), we get
 	\[
 	\begin{aligned}
 	&\frac {1}{4}\lambda^{2}\left| \underline{\beta}\right|^{2}\left(\sin ^{4}\alpha_1+  2\sin^{2}\alpha_1\sin^{2}\alpha_2+sin^{4}\alpha_2\right)\\
 	=&\frac {1}{4}\lambda^{2}\left| \underline{\beta}\right| ^{2}\left( \sin^{4}\alpha_1+ sin\alpha_1sin^{3}\alpha_2+sin^{3}\alpha_1sin\alpha_2+\sin^{4}\alpha_2\right).
 	\end{aligned}
 	\]
 	That is \(\sin \alpha_1 =\sin \alpha_2\).
 	This completes the proof.
 	
 	\hfill $\square$

 	Moreover, since the standard covariance satisfies \(\left|Cov_{x_{k}w _{k}}\right|\leq COV_{x_{k}w _{k}} \), the inequality can be appropriately relaxed. Based on this relation, we have Corollary 4.10 below.
\begin{cor}
	Let \(f(\textbf{x})=\rho (\textbf{x}) e^{\textbf{u}(\textbf{x}) \theta(\textbf{x})}\). If \(f(\textbf{x}) \in L^{2}(\mathbb{R}, \mathbb{H})\), \(x_{k} f(\textbf{x})\), \(\frac{\partial}{\partial x_{k}} f(\textbf{x}) \in L^{2}(\mathbb{R}, \mathbb{H})\)  and \(\|f\|_{L^{2}}=1\), let \(e_{1}=i; e_{2}=j\), then
	\begin{equation}
	\begin{aligned}
	&\int_{\mathbb{R}^{2}}\left| \textbf{x}\right| ^{2} |f(\textbf{x})|^{2} d \textbf{x}\int_{\mathbb{R}^{2}}\left| \textbf{w}\right| ^{2} |\mathcal{F}_{\alpha_{1},\alpha_{2}}\{f\}(\textbf{w})|^{2} d \textbf{w} \\
	\geq& \frac{1}{4} \left| \sin\alpha_{1}+\sin\alpha_{2}\right| ^{2}+ \left| Cov_{\textbf{x}\textbf{w},\alpha}\right| ^{2}+\int_{\mathbb{R}^{2}} \sum_{k=1}^{2} x_{k}^{2} \rho^{2} d\textbf{x} \int_{\mathbb{R}^{2}} \sum_{k=1}^{2}\cos^{2}\alpha_{k} x_{k}^{2} \rho^{2} d\textbf{x}\\
	& -\!\!\int_{\mathbb{R}^{2}} \!\!\sum_{k=1}^{2} x_{k}^{2} \rho^{2} d\textbf{x} \! \int_{\mathbb{R}^{2}}\!\!\sum_{k=1}^{2}2\sin\alpha_{k}\cos\alpha_{k}x_{k}\rho^{2}Sc \left[ e_{k}\!\left( \!\frac{\partial}{\partial x_{k}} e^{\textbf{u}(\textbf{x}) \theta(\textbf{x})} \!\!\right)\!\! \left( e^{-\textbf{u}(\textbf{x}) \theta(\textbf{x})} \right)\! \right]\! d\textbf{x},
	\end{aligned}
	\end{equation}
	where the absolute covariance is defined by
	\[
	Cov_{\textbf{x}\textbf{w},\alpha}:=\sum_{k=1}^{2}\int _{\mathbb{R}^{2}} x_{k}\sin\alpha_{k}\left( \frac{\partial}{\partial x_{k}} e^{\textbf{u}(\textbf{x}) \theta(\textbf{x})} \right) \left( e^{-\textbf{u}(\textbf{x}) \theta(\textbf{x})} \right)  \rho ^{2}(\textbf{x})d\textbf{x}.
	\]
\end{cor}

\textbf{ Proof of Theorem 1.2.}
By Theorem 4.4, we have
	\[
\begin{aligned}
&\int_{\mathbb{R}^{2}}\left| \textbf{w}\right| ^{2} |\mathcal{F}_{\alpha_{1},\alpha_{2}}\{f\}(\textbf{w})|^{2} d \textbf{w}\int_{\mathbb{R}^{2}}\left| \textbf{y}\right| ^{2} |\mathcal{F}_{\beta_{1},\beta_{2}}\{f\}(\textbf{y})|^{2} d \textbf{y} \\
=& \int_{\mathbb{R}^{2}} \sum_{k=1}^{2} w_{k}^{2} |\mathcal{F}_{\alpha_{1},\alpha_{2}}\{f\}(\textbf{w})|^{2} d\textbf{w}   \int_{\mathbb{R}^{2}} \sum_{k=1}^{2} y_{k}^{2} |\mathcal{F}_{\beta_{1},\beta_{2}}\{f\}(\textbf{y})|^{2} d\textbf{y} \\
=&\left( I_{1}+I_{2}+I_{3}+I_{4}\right) \left( I_{5}+I_{6}+I_{7}+I_{8}\right) \\
 =&I_{1}I_{5}+I_{1}I_{6}+I_{1}I_{7}+I_{1}I_{8}+I_{2}I_{5}+I_{2}I_{6}+I_{2}I_{7}+I_{2}I_{8}\\
 &+I_{3}I_{5}+I_{3}I_{6}+I_{3}I_{7}+I_{3}I_{8}+I_{4}I_{5}+I_{4}I_{6}+I_{4}I_{7}+I_{4}I_{8}.
 \end{aligned}
 \]
 where \(I_{1},I_{2}\) etc. are given in Theorem 1.2.
 	 Applying the continuous case of the Cauchy-Schwarz inequality and its discrete case, we obtain these inequalities.
 	
 	 \[
 	 \begin{aligned}
 	 I_{1}I_{5}
 	 &= \int_{\mathbb{R}^{2}} \sum_{k=1}^{2}sin^{2} \alpha_k  \left( \frac{\partial}{\partial x_{k}} \rho(\textbf{x})\right)  ^{2} d\textbf{x}  \int_{\mathbb{R}^{2}} \sum_{k=1}^{2}sin^{2} \beta_k  \left( \frac{\partial}{\partial x_{k}} \rho(\textbf{x})\right)  ^{2} d\textbf{x}  \\
 	 &\geq \left| \int_{\mathbb{R}^{2}}  \sum_{k=1}^{2}sin \alpha_ksin \beta_k  \left( \frac{\partial}{\partial x_{k}} \rho(\textbf{x})\right)  ^{2}d\textbf{x}\right| ^{2}.
 	 \end{aligned}
 	 	 \]
 	By (4.4), we have
 	 	  \[
 	 	 \begin{aligned}
 	 	 I_{1}I_{6}
 	 	 &= \int_{\mathbb{R}^{2}} \sum_{k=1}^{2}sin^{2} \alpha_k  \left( \frac{\partial}{\partial x_{k}} \rho(\textbf{x})\right)  ^{2} d\textbf{x}  \int_{\mathbb{R}^{2}} \sum_{k=1}^{2}\cos^{2}\beta_{k} x_{k}^{2} \rho^{2} d\textbf{x}\\
 	 	 &\geq\frac{1}{4}\left| \sum_{k=1}^{2}sin \alpha_kcos \beta_k\right| ^{2}
 	 	 \end{aligned}
 	 	 \]
 	 	 and
 	 	  \[
 	 	 \begin{aligned}
 	 	 I_{1}I_{7}
 	 	 &= \!\!\int_{\mathbb{R}^{2}} \!\!\sum_{k=1}^{2}sin^{2} \alpha_k\!  \left( \!\frac{\partial}{\partial x_{k}} \rho(\textbf{x})\!\right)  ^{2}\!\! d\textbf{x}  \!\!\int_{\mathbb{R}^{2}}\!\! \sum_{k=1}^{2}sin^{2} \beta_k \rho^{2}\!\left|\left( \frac{\partial}{\partial x_{k}} e^{\textbf{u}(\textbf{x}) \theta(\textbf{x})} \!\!\right)\!\! \left( e^{-\textbf{u}(\textbf{x}) \theta(\textbf{x})} \right)\right| ^{2}\!\!d\textbf{x}\\
 	 	 &\geq\left| \int_{\mathbb{R}^{2}}  \sum_{k=1}^{2}sin \alpha_ksin \beta_k\rho(\textbf{x})  \left( \frac{\partial}{\partial x_{k}} \rho(\textbf{x})\right)  \left| \left( \frac{\partial}{\partial x_{k}} e^{\textbf{u}(\textbf{x}) \theta(\textbf{x})} \right) \left( e^{-\textbf{u}(\textbf{x}) \theta(\textbf{x})} \right)\right|d\textbf{x}\right| ^{2}.
 	 	 \end{aligned}
 	 	 \]
 	Using the same method as above, we can obtain the remaining terms as follows.
 	\[
 	\begin{aligned}
 	I_{2}I_{5}
 	&= \int_{\mathbb{R}^{2}} \sum_{k=1}^{2}\cos^{2}\alpha_{k} x_{k}^{2} \rho^{2} d\textbf{x} \int_{\mathbb{R}^{2}} \sum_{k=1}^{2}sin^{2} \beta_k  \left( \frac{\partial}{\partial x_{k}} \rho(\textbf{x})\right)  ^{2} d\textbf{x}  \\
  &\geq \frac{1}{4}\left| \sum_{k=1}^{2}cos \alpha_ksin \beta_k\right| ^{2}.
 	\end{aligned}
 	\]
 	And
 	\[
 	\begin{aligned}
 	&I_{2}I_{6}
 	\geq \left| \int_{\mathbb{R}^{2}}  \sum_{k=1}^{2}cos \alpha_kcos \beta_k   x_{k}^{2}\rho^{2}(\textbf{x})d\textbf{x}\right| ^{2},\\
 &I_{2}I_{7}
  \geq \left| \int_{\mathbb{R}^{2}}  \sum_{k=1}^{2}cos \alpha_ksin \beta_kx_{k}\rho^{2}(\textbf{x}) \left|\left( \frac{\partial}{\partial x_{k}} e^{\textbf{u}(\textbf{x}) \theta(\textbf{x})} \right) \left( e^{-\textbf{u}(\textbf{x}) \theta(\textbf{x})} \right)\right|d\textbf{x}\right| ^{2},\\
 &	I_{3}I_{5}
 	\geq \left| \int_{\mathbb{R}^{2}}  \sum_{k=1}^{2}sin \alpha_ksin \beta_k\rho(\textbf{x}) \left( \frac{\partial}{\partial x_{k}} \rho(\textbf{x})\right) \left|\left( \frac{\partial}{\partial x_{k}} e^{\textbf{u}(\textbf{x}) \theta(\textbf{x})} \right) \left( e^{-\textbf{u}(\textbf{x}) \theta(\textbf{x})} \right)\right|d\textbf{x}\right| ^{2},\\
 	&I_{3}I_{6}
 	\geq \left| \int_{\mathbb{R}^{2}}  \sum_{k=1}^{2}sin \alpha_kcos \beta_kx_{k}\rho^{2}(\textbf{x}) \left|\left( \frac{\partial}{\partial x_{k}} e^{\textbf{u}(\textbf{x}) \theta(\textbf{x})} \right) \left( e^{-\textbf{u}(\textbf{x}) \theta(\textbf{x})} \right)\right|d\textbf{x}\right| ^{2},\\
 	&I_{3}I_{7}
 	\geq \left| \int_{\mathbb{R}^{2}}  \sum_{k=1}^{2}sin \alpha_ksin \beta_k\rho^{2}(\textbf{x}) \left|\left( \frac{\partial}{\partial x_{k}} e^{\textbf{u}(\textbf{x}) \theta(\textbf{x})} \right) \left( e^{-\textbf{u}(\textbf{x}) \theta(\textbf{x})} \right)\right|^{2}d\textbf{x}\right| ^{2}.
 	\end{aligned}
 	\] 	
 The above terms plus \(I_{1}I_{8}, I_{2}I_{8}, I_{3}I_{8},I_{4}I_{5}, I_{4}I_{6}, I_{4}I_{7}, I_{4}I_{8}\), we have the final inequality in Theorem 1.2.

\section{Example}
In this section, we provide an example to show that there indeed exists a function for which the equality holds in Theorem 4.7 and Theorem 1.1.

\begin{ex}
		Consider the function
	\[
	f(\textbf{x})=\frac {1}{\sqrt{\pi \lambda\sin\alpha}}e^{-\frac{1}{2\lambda} \frac{1}{\sin\alpha}\left| \textbf{x}\right| ^{2} }e^{\beta\left| \textbf{x}\right| ^{2}},
	\]
	where $\lambda$ is a positive number, \(\sin \alpha > 0 \) and \(\beta\in S^2\) is a pure quaternion.
	
	Computing directly, we have
	\begin{equation}
	\begin{aligned}
	\int_{\mathbb{R}^{2}} x_{k}^{2}|f(\textbf{x})|^{2} d \textbf{x}
	&=\frac {1}{\pi \lambda\sin\alpha}\int_{\mathbb{R}^{2}}x_{k}^{2}e^{-\frac{1}{\lambda}\frac{1}{\sin\alpha}\left| \textbf{x}\right|^{2}} d\textbf{x}\\
	&=\frac {1}{2}\lambda\sin \alpha.
	\end{aligned}
	\end{equation}
	Since \(\rho (\textbf{x})=\frac {1}{\sqrt{\pi \lambda\sin\alpha}}e^{-\frac{1}{2\lambda} \frac{1}{\sin\alpha}\left| \textbf{x}\right| ^{2} }\), directly computing through Theorem 4.5, we have
	
	\begin{equation}
	\begin{aligned}
	&\int_{\mathbb{R}^{2}} w_{k}^{2}|\mathcal{F}_{\alpha_{1},\alpha_{2}}\{f\}(\textbf{w})|^{2} d \textbf{w}\\
	=&\sin^{2}\alpha\int_{\mathbb{R}^{2}}\left[\frac{\partial}{\partial x_{k}} \rho(\textbf{x})\right]^{2} d \textbf{x}+\cos^{2}\alpha\int_{\mathbb{R}^{2}} x_{k}^{2}\rho^{2}(\textbf{x})d\textbf{x} \\
	& +\sin^{2}\alpha \int_{\mathbb{R}^{2}} \rho^{2}(\textbf{x})\left|\left(\frac{\partial}{\partial x_{k}} e^{\textbf{u}(\textbf{x}) \theta(\textbf{x})}\right)\left(e^{-\textbf{u}(\textbf{x}) \theta(\textbf{x})}\right)\right|^{2} d \textbf{x}\\
	&-2\sin\alpha\cos\alpha\int_{\mathbb{R}^{2}} x_{k}\rho^{2}(\textbf{x})Sc\left[   \left( \frac{\partial}{\partial x_{k}} e^{\textbf{u}(\textbf{x}) \theta(\textbf{x})} \right) \left( e^{-\textbf{u}(\textbf{x}) \theta(\textbf{x})} \right)\right] d\textbf{x}\\
	=&\frac {1}{2\lambda}\sin \alpha+2\lambda\sin ^{3}\alpha+\frac {\lambda}{2}\cos ^{2}\alpha\sin \alpha-2\lambda\cos \alpha\sin ^{2}\alpha.
		\end{aligned}
	\end{equation}
	 By (5.1) and (5.2), We can obtain the left-hand side of (4.1).
	\[
	\begin{aligned}
		& \int_{\mathbb{R}^{2}} x_{k}^{2}|f(\textbf{x})|^{2} d \textbf{x} \int_{\mathbb{R}^{2}} w_{k}^{2}|\mathcal{F}_{\alpha_{1},\alpha_{2}}\{f\}(\textbf{w})|^{2} d \textbf{w} \\
		=&\frac{1}{4}\sin^{2}\alpha+\sin^{2}\alpha\left( \lambda\sin\alpha\right)  ^{2}+\frac{1}{4}\lambda^{2}\sin^{2}\alpha\cos^{2}\alpha-\lambda^{2}\sin^{3}\alpha\cos\alpha.
			\end{aligned}
		\]
		Now we compute the right-hand side of (4.1).
		By computing, it is easy to calculate that \(COV_{x_{k}w_{k}}=\lambda\sin \alpha, Cov_{x_{k}w_{k}}=\beta\lambda\sin \alpha\).
		\[
		\int_{\mathbb{R}^{2}}x_{k}^{2}\rho^{2}(\textbf{x})d\textbf{x}=\frac{1}{2}\lambda\sin\alpha,
		\]
		\[
		\int_{\mathbb{R}^{2}} x_{k}\rho^{2}(\textbf{x})Sc\left[ i\left( \frac{\partial}{\partial x_{k}} e^{\textbf{u}(\textbf{x}) \theta(\textbf{x})} \right) \left( e^{-\textbf{u}(\textbf{x}) \theta(\textbf{x})} \right) \right]d\textbf{x}=\lambda\sin\alpha.
		\]
		We can observe that the left-hand side is equal to the right-hand side of (4.1). That is
		\[
		\begin{aligned}
		& \int_{\mathbb{R}^{2}} x_{k}^{2}|f(\textbf{x})|^{2} d \textbf{x} \int_{\mathbb{R}^{2}} w_{k}^{2}|\mathcal{F}_{\alpha_{1},\alpha_{2}}\{f\}(\textbf{w})|^{2} d \textbf{w} \\
		=&\frac{1}{4}\sin^{2}\alpha+\sin^{2}\alpha\left( \lambda\sin\alpha\right)  ^{2}+\frac{1}{4}\lambda^{2}\sin^{2}\alpha\cos^{2}\alpha-\lambda^{2}\sin^{3}\alpha\cos\alpha\\
		=&\frac{1}{4} \sin^{2}\alpha+ \sin^{2}\alpha COV_{x_{k} w_{k}}^{2}+ \cos^{2}\alpha\left( \int_{\mathbb{R}^{2}}x_{k}^{2}\rho(\textbf{x})^{2}d\textbf{x}\right) ^{2}
			\end{aligned}
		\]
		\[
		\begin{aligned}
		&  -\!2\sin\alpha\cos\alpha\int_{\mathbb{R}^{2}}\!\!x_{k}^{2}\rho^{2}(\textbf{x})d\textbf{x}\!\!\int_{\mathbb{R}^{2}} \!\!x_{k}\rho^{2}(\textbf{x})Sc\left[ i\left( \frac{\partial}{\partial x_{k}} e^{\textbf{u}(\textbf{x}) \theta(\textbf{x})} \right) \left( e^{-\textbf{u}(\textbf{x}) \theta(\textbf{x})} \right) \right]d\textbf{x}\\
		=&\frac{1}{4} \sin^{2}\alpha+ \sin^{2}\alpha\left| \beta\lambda\sin\alpha\right| ^{2}+ \frac{1}{4}\lambda^{2}\sin^{2}\alpha\cos^{2}\alpha-\lambda^{2}\sin^{3}\alpha\cos\alpha\\
		=&\frac{1}{4} \sin^{2}\alpha+ \sin^{2}\alpha\left| Cov_{x_{k}w_{k}}\right| ^{2}+ \cos^{2}\alpha\left( \int_{R^{2}}x_{k}^{2}\rho(\textbf{x})^{2}d\textbf{x}\right) ^{2}\\
		&  -2\sin\alpha\cos\alpha\!\int_{\mathbb{R}^{2}}\!\!x_{k}^{2}\rho^{2}(\textbf{x})d\textbf{x}\!\!\int_{\mathbb{R}^{2}}\!\! x_{k}\rho^{2}(\textbf{x})Sc\left[i \left( \frac{\partial}{\partial x_{k}} e^{\textbf{u}(\textbf{x}) \theta(\textbf{x})} \right) \left( e^{-\textbf{u}(\textbf{x}) \theta(\textbf{x})} \right) \right]d\textbf{x}.
		\end{aligned}
	\]
	Next we compute the left-hand side of (1.5).
	 By (5.1) and (5.2), we get
	\[
	\begin{aligned}
	&\int_{\mathbb{R}^{2}}\left| \textbf{x}\right| ^{2} |f(\textbf{x})|^{2} d \textbf{x}\int_{\mathbb{R}^{2}}\left| \textbf{w}\right| ^{2} |\mathcal{F}_{\alpha_{1},\alpha_{2}}\{f\}(\textbf{w})|^{2} d \textbf{w} \\
	=& \int_{\mathbb{R}^{2}} \sum_{k=1}^{2} x_{k}^{2} \rho^{2} d\textbf{x}  \int_{\mathbb{R}^{2}} \sum_{k=1}^{2} w_{k}^{2} |\mathcal{F}_{\alpha_{1},\alpha_{2}}\{f\}(\textbf{w})|^{2} d\textbf{w} \\
	=&\sin ^{2}\alpha+4\lambda^{2}\sin^{4} \alpha+\lambda^{2}\sin^{2} \alpha\cos^{2} \alpha-4\lambda^{2}\sin^{3} \alpha\cos \alpha.
	\end{aligned}
	\]
		Now we compute the right-hand side of (1.5).By computing, it is easy to see that \(COV_{\textbf{x}\textbf{w}}=2\lambda\sin ^{2}\alpha, Cov_{\textbf{x}\textbf{w}}=2\beta\lambda\sin ^{2}\alpha\), and
		\[
		\int_{\mathbb{R}^{2}} \sum_{k=1}^{2} x_{k}^{2} \rho^{2} d\textbf{x}=\lambda\sin\alpha ,
		\]
		\[
		\int_{\mathbb{R}^{2}}\!\!\sum_{k=1}^{2}x_{k}\rho^{2}Sc \left[e_{k} \left( \frac{\partial}{\partial x_{k}} e^{\textbf{u}(\textbf{x}) \theta(\textbf{x})} \right) \left( e^{-\textbf{u}(\textbf{x}) \theta(\textbf{x})} \right) \right]d\textbf{x}=2\lambda\sin\alpha.
		\]
		We can observe that the left-hand side is equal to the right-hand side of (1.5). That is
	\[
	\begin{aligned}
	&\int_{\mathbb{R}^{2}}\left| \textbf{x}\right| ^{2} |f(\textbf{x})|^{2} d \textbf{x}\int_{\mathbb{R}^{2}}\left| \textbf{w}\right| ^{2} |\mathcal{F}_{\alpha_{1},\alpha_{2}}\{f\}(\textbf{w})|^{2} d \textbf{w} \\
	=& \int_{\mathbb{R}^{2}} \sum_{k=1}^{2} x_{k}^{2} \rho^{2} d\textbf{x}  \int_{\mathbb{R}^{2}} \sum_{k=1}^{2} w_{k}^{2} |\mathcal{F}_{\alpha_{1},\alpha_{2}}\{f\}(\textbf{w})|^{2} d\textbf{w} \\
	=&\sin ^{2}\alpha+4\lambda^{2}\sin^{4} \alpha+\lambda^{2}\sin^{2} \alpha\cos^{2} \alpha-4\lambda^{2}\sin^{3} \alpha\cos \alpha\\
	=&\frac {1}{4}\left( \sin\alpha+\sin \alpha\right)^{2}+\left( 2\lambda\sin^{2}\alpha\right) ^{2}\!+\!\cos^{2} \alpha\lambda^{2}\sin^{2} \alpha\!-\!2\sin \alpha\cos \alpha \cdot \lambda\sin \alpha\cdot 2\lambda\sin \alpha\\
	=&\frac {1}{4}\left( \sin \alpha+\sin \alpha\right)^{2}+COV_{\textbf{x}\textbf{w}}^{2}+\cos^{2}\alpha\left(\int_{\mathbb{R}^{2}} \sum_{k=1}^{2} x_{k}^{2} \rho^{2} d\textbf{x} \right)^{2}\\
	&-\!\!2\sin\alpha\cos\alpha\!\!\int_{\mathbb{R}^{2}} \!\!\sum_{k=1}^{2} x_{k}^{2} \rho^{2} d\textbf{x} \!\! \int_{\mathbb{R}^{2}}\!\!\sum_{k=1}^{2}x_{k}\rho^{2}Sc \left[e_{k} \left( \frac{\partial}{\partial x_{k}} e^{\textbf{u}(\textbf{x}) \theta(\textbf{x})} \right) \left( e^{-\textbf{u}(\textbf{x}) \theta(\textbf{x})} \right) \right]d\textbf{x} \\
		=&\frac {1}{4}\left( \sin \alpha+\sin \alpha\right)^{2}\!+\!\left|  2\beta \lambda\sin^{2}\alpha\right|  ^{2}\!+\!\cos^{2} \alpha\lambda^{2}\sin^{2} \alpha-\!2\sin \alpha\cos \alpha \cdot \lambda\sin \alpha\cdot 2\lambda\sin \alpha
		\end{aligned}
		\]
		\[
		\begin{aligned}
	=&\frac {1}{4}\left( \sin \alpha+\sin \alpha\right)^{2}+\left| Cov_{\textbf{x}\textbf{w}}\right| ^{2}+\cos^{2}\alpha\left(\int_{\mathbb{R}^{2}} \sum_{k=1}^{2} x_{k}^{2} \rho^{2} d\textbf{x} \right)^{2}\\
	&-\!\!2\sin\alpha\cos\alpha\!\!\int_{\mathbb{R}^{2}} \!\sum_{k=1}^{2} x_{k}^{2} \rho^{2} d\textbf{x} \! \int_{\mathbb{R}^{2}}\!\sum_{k=1}^{2}x_{k}\rho^{2}Sc \!\left[e_{k} \left( \frac{\partial}{\partial x_{1}} e^{\textbf{u}(\textbf{x}) \theta(\textbf{x})} \right) \left( e^{-\textbf{u}(\textbf{x}) \theta(\textbf{x})} \right) \right]\!\!d\textbf{x} .
	\end{aligned}
	\]
	Note that, in this case, the uncertainty principle becomes an equality.
	\end{ex}
 \section*{Acknowledgement}
 Haipan Shi is supported by Natural Science Foundation of Zhejiang Province (Grant No.LQN25A010014, Grant No.LMS26A010001) and  China Scholarships Council under No. 202509710010.
 Xiaomin Tang is supported by Natural Science Foundation of Zhejiang Province (Grant No.LZ24A010004).



\begin{flushleft}
	\vspace{6pt}
	Ke Cui\\
	School of Science\\
	Huzhou Normal University\\
	Wuxing District, Huzhou City, Zhejiang Province, 313000, P. R. China\\
	E-mail:cuike1015@163.com\\
	\vspace{4pt}
	Haipan Shi\\
	School of Science\\
	Huzhou Normal University\\
	Wuxing District, Huzhou City, Zhejiang Province, 313000, P. R. China\\
	E-mail:shihaipan226@163.com\\
	\vspace{4pt}
	Xiaomin Tang\\
	School of Science\\
	Huzhou Normal University\\
	Wuxing District, Huzhou City, Zhejiang Province, 313000, P. R. China\\
	E-mail:txm@zjhu.edu.cn
\end{flushleft}

\begin{thebibliography}{10}
	\bibitem{1} Bahri M., Hitzer E., Hayashi A., Ashino R.. An uncertainty principle for quaternion Fourier transform. Comput. Math. Appl, 2008, 56: 2398-2410.
	\bibitem{2} Condon E. U.. Remarks on uncertainty principles. Science, 1929, 69(1796): 573.
	\bibitem{3} Cohen L.. Time-Frequency Analysis: Theory and Applications. Prentice Hall
	Inc., Upper Saddle River, 1995.
	\bibitem{4} Dang P., Deng G. T., Qian T.. A sharper uncertainty principle. J. Funct. Anal, 2013.
	\bibitem{5} Gabor D.. Theory of communication. Part 1: the analysis of information. J. Inst. Elect r. Eng-Part III: Radio Commun. Eng, 1946, 93(26): 429-441.
	\bibitem{6} Grigoryan A. M., Jenkinson J., Agaian S. S.. Quaternion Fourier transform based alpha-rooting method for color image measurement and enhancement. Signal Processing, 2015, 109: 269-289.
	\bibitem{7} Grigoryan A. M., John A., Agaian S. S.. Alpha-rooting color image enhancement method by two-side 2-D quaternion discrete Fourier transform followed by spatial transformation. IJACEEE, 2018.
	\bibitem{8} Heisenberg W.. Uber den anschaulichen Inhalt der quanten theoretischen Kinematikund Mechanik. Z. Phys, 1927, 43: 172-198.
	\bibitem{9} Hitzer E. M. S.. Directional uncertainty principle for quaternion Fourier transform. Adv. Appl. Cli?ord Algebr, 2010, 20: 271-284.
	\bibitem{10} Li Z., Shi H. P., Qiao Y.. Two-sided fractional quaternion Fourier transform and its application. J Inequal Appl. 2021, 2021: 121.
	\bibitem{11} Nicewarner K. E., Sanderson A. C.. A General Representation for Orientational Uncertainty Using Random Unit Quaternions. Proceedings of IEEE International Conference on Robotics and Automation, 1994, 1161-1168.
	\bibitem{12} Sudbery, A.. Quaternionic analysis. Math. Proc. Camb. Philos. Soc., 1979, 85: 199-225.
	\bibitem{13} Took C. C., Mandic D. P.. The quaternion LMS algorithm for adaptive ?ltering of hypercomplex processes. Signal Processing, 2009, 57(4): 1316-1327.
	\bibitem{14} Took C. C., Mandic D.P.. Augmented second-order statistics of quaternion random signals. Signal Processing, 2011, 91(2): 214-224.
	\bibitem{15} Weyl H.. The Theory of Groups and Quantum Mechanics. E. P. Dutton and Co., 1931.
	\bibitem{16} Yang Y., Dang P., Qian T.. Tighter Uncertainty Principles Based on Quaternion Fourier Transform. Adv. Appl. Clifford Algebras, 2016, 26: 479-497.
\end{thebibliography}
\end{document}